\documentclass[letterpaper,12pt]{article}
\usepackage{latexsym,amssymb,amsfonts,amsmath,amsthm,nccmath,enumitem,setspace,graphicx,float}
\usepackage[dvipsnames]{xcolor}
\usepackage[hidelinks]{hyperref}
\usepackage[margin=2cm]{geometry}
\usepackage{subcaption}
\usepackage{multicol}
\usepackage{dsfont}
\usepackage{bm}
\usepackage{float}
\usepackage[normalem]{ulem}
\newtheorem{theorem}{Theorem}[section]
\newtheorem{corollary}[theorem]{Corollary}
\newtheorem{assumption}[theorem]{Assumption}
\newtheorem{conj}[theorem]{Conjecture}
\newtheorem{lemma}[theorem]{Lemma}
\newtheorem{proposition}[theorem]{Proposition}
\newtheorem{remark}[theorem]{Remark}
\newtheorem{notation}[theorem]{Notation}

\theoremstyle{definition}
\newtheorem{definition}[theorem]{Definition}

\usepackage{hyperref}

\usepackage[maxfloats=256]{morefloats}
\definecolor{color1}{rgb}{0,0, 0.0}
\definecolor{color8}{rgb}{0,0, 0.0}
\definecolor{color2}{rgb}{0.55, 0.55, 0.55}
\definecolor{color4}{rgb}{0.45, 0.66, 0.76}

\usepackage{tikz}
\tikzstyle{vertex}=[circle, draw=black, fill=black, minimum size=2pt, inner sep=2]
\usetikzlibrary{arrows, snakes,arrows.meta,calc,decorations.markings,math,arrows.meta}

\makeatletter

\tikzset{
  on each segment/.style={
    decorate,
    decoration={
      show path construction,
      moveto code={},
      lineto code={
        \path [#1]
        (\tikzinputsegmentfirst) -- (\tikzinputsegmentlast);
      },
      curveto code={
        \path [#1] (\tikzinputsegmentfirst)
        .. controls
        (\tikzinputsegmentsupporta) and (\tikzinputsegmentsupportb)
        ..
        (\tikzinputsegmentlast);
      },
      closepath code={
        \path [#1]
        (\tikzinputsegmentfirst) -- (\tikzinputsegmentlast);
      },
    },
  },
  mid arrow/.style={postaction={decorate,decoration={
        markings,
        mark=at position .5 with {\arrow[scale=1.2,#1]{stealth}}
      }}},
}

\usetikzlibrary{decorations.markings}
\tikzset{
    set arrow inside/.code={\pgfqkeys{/tikz/arrow inside}{#1}},
    set arrow inside={end/.initial=>, opt/.initial=},
    /pgf/decoration/Mark/.style={
        mark/.expanded=at position #1 with
        {
            \noexpand\arrow[\pgfkeysvalueof{/tikz/arrow inside/opt}]{\pgfkeysvalueof{/tikz/arrow inside/end}}
        }
    },
    arrow inside/.style 2 args={
        set arrow inside={#1},
        postaction={
            decorate,decoration={
                markings,Mark/.list={#2}
            }
        }
    },
}

\makeatother

\let\oldproofname=\proofname
\renewcommand{\proofname}{\rm\bf{\oldproofname}}

\title{Decompositions of the wreath product of certain directed graphs into directed hamiltonian cycles}
\author{A. Lacaze-Masmonteil\footnote{Email: alk004@uregina.ca. Mailing address: Department of Mathematics and Statistics, University of Regina, 3737 Wascana Pkwy, Regina, SK S4S 0A2, Canada}, University of Regina }

\begin{document}
\maketitle \baselineskip 17pt

\begin{center}
{\bf Abstract}
\end{center}

We affirm several special cases of a conjecture that first appears in Alspach et al.~(1987) which stipulates that the wreath (lexicographic) product of two hamiltonian decomposable directed graphs is also hamiltonian decomposable. Specifically, we show that the wreath product of  hamiltonian decomposable directed graph $G$, such that $|V(G)|$ is even and $|V(G)|\geqslant 3$, with a directed $m$-cycle such that $m \geqslant 4$ or the complete symmetric directed graph on $m$ vertices such that $m\geqslant 3$, is hamiltonian decomposable. We also show  the wreath product of a directed $n$-cycle, where $n$ is even, with a directed $m$-cycle, where $m \in \{2,3\}$, is not hamiltonian decomposable.

\section{Introduction}\label{S:intro}

In this paper, we determine when the wreath (lexicographic) product of certain hamiltonian decomposable directed graphs (digraphs for short) is also hamiltonian decomposable. The \textit{wreath product} of digraph $G$ with $H$, denoted $G \wr H$, is the digraph on vertex set $V(G) \times V(H)$ such that $((g_1, h_1), (g_2, h_2)) \in A(G \wr H)$ if and only if $(g_1, g_2)\in A(G)$, or $g_1=g_2$ and $(h_1, h_2)\in A(H)$.  If a (directed) graph admits a decomposition into (directed) hamiltonian cycles, then this (directed) graph is \textit{hamiltonian decomposable}.  It is well-known that the complete graph $K_n$ is hamiltonian decomposable for all odd $n$, as shown by Walecki in \cite{Wale}, while the complete symmetric directed graph $K_n^*$ is hamiltonian decomposable if and only if $n \not\in \{4,6\}$ \cite{Berm2, Tilson}. 

The question of determining whether a particular product of two hamiltonian decomposable (directed) graphs is also hamiltonian decomposable has a rich history. The wreath product of two hamiltonian decomposable graphs is hamiltonian decomposable  \cite{HamLex}; likewise for the strong product \cite{Fan, Zhou}.  An analogous statement has also been proved for the Cartesian product \cite{Stong} under mild additional assumptions. Regarding the categorical (tensor) product, Bermond \cite{Berm} has shown that, if $G$ and $H$ are hamiltonian decomposable graphs and at least one of $|V(G)|$ and $|V(H)|$ is odd, then the categorical product of $G$ with $H$ is also hamiltonian decomposable. In \cite{Bala+}, Balakrishnan et al.~show that the categorical product of two complete graphs of order at least three is hamiltonian decomposable. 

There are few results on the analogous problem for digraphs. In \cite{Kea}, Keating shows that the Cartesian product of a directed $n$-cycle with a directed $m$-cycle is hamiltonian decomposable if and only if there exist positive integers $s_1$ and $s_2$ such that $\textrm{gcd}(m, n)=s_1+s_2$ and $\textrm{gcd}(mn, s_1s_2)=1$. Furthermore, Paulraja and Sivasankar \cite{Praja2} show that the categorial product of specific classes of hamiltonian decomposable digraphs is hamiltonian decomposable. Regarding the wreath product of two hamiltonian decomposable digraphs, the following conjecture first appeared in \cite{CD1} in 1987.

\begin{conj} \cite{CD1}
\label{conj:main}
Let $G$ and $H$ be hamiltonian decomposable directed graphs. The wreath product $G \wr H$  is also hamiltonian decomposable. 
\end{conj}

In \cite{Ng}, Ng  identifies two exceptions to Conjecture \ref{conj:main}. Namely, if $G$ is a directed $n$-cycle such that $n$ is odd and $H\in \{\overline{K}_2, K^*_2\}$, then the wreath product of $G$ with $H$ is not hamiltonian decomposable. However, if $H$ is a digraph on at least three vertices, then Ng affirms Conjecture \ref{conj:main} in the following case. 

\begin{theorem} \cite{Ng}
\label{thm:ng}
Let $G$ and $H$ be two hamiltonian decomposable digraph such that $|V(G)|$ is odd and $|V(H)|\geqslant 3$. Then $G \wr H$ is hamiltonian decomposable.
\end{theorem}

Crucial to Ng's approach is the following lemma, implicitly proved in the proof of Proposition 1 of \cite{Ng}, which can be viewed as a reduction step. This lemma is also key to our approach. 

\begin{lemma} \cite{Ng}
\label{lem:redCn}
Let $G$ and $H$ be two hamiltonian decomposable digraphs such that $|V(G)|=n$ and $|V(H)|>2$.  If $\vec{C}_n \wr H$ is hamiltonian decomposable, then $G \wr H$ is also hamiltonian decomposable.   
\end{lemma}

Theorem \ref{thm:ng} and Lemma \ref{lem:redCn} imply that it suffices to concentrate on the case of Conjecture \ref{conj:main} where $G=\vec{C}_n$ with $n$ even. In \cite{Me}, the author provides a near-complete solution to Conjecture \ref{conj:main} when $|V(G)|$ is even and $H$ is not a complete graph nor a directed cycle. 

\begin{theorem} \cite{Me}
\label{thm:al}
Let $G$ and $H$ be two hamiltonian decomposable digraphs such that $|V(G)|$ is even, $|V(H)|>3$, $|V(H)|=m$, and $H \not\in \{K^*_m, \vec{C}_m\}$. Then $G \wr H$ is hamiltonian decomposable except possibly when $G$ is a directed cycle, $|V(H)|$ is even, and $H$ admits a decomposition into an odd number of directed hamiltonian cycles.
\end{theorem}

Using a very different approach from \cite{Me}, we complement Theorem \ref{thm:al} for cases of Conjecture \ref{conj:main} in which $H$ is a complete symmetric digraph or a directed cycle. Our contribution to Conjecture \ref{conj:main} is summarized in Theorems \ref{thm:1} and \ref{thm:2} below. 

\begin{theorem}
\label{thm:1}
Let $G$ be a hamiltonian decomposable digraph of even order $n$.  The digraph $G \wr K^*_{m}$ is hamiltonian decomposable if and only if $(n,m) \neq(2,3)$ and $m \neq 2$.
\end{theorem}

\begin{theorem}
\label{thm:2}
Let $G$ be a hamiltonian decomposable digraph of even order $n$. The digraph $G\wr \vec{C}_{m}$ is hamiltonian decomposable in each of the following cases:
\begin{enumerate} [label=\textbf{(S\arabic*)}]
\item $m$ is even, $m \geqslant 4$, and $n \geqslant 4$;
\item $m=4$ and $n=2$;
\item $m$ is odd, and $m\geqslant 5$.
\end{enumerate}
Furthermore, if $G=\vec{C}_n$ and $m\in \{2,3\}$, then $G\wr \vec{C}_m$ is not hamiltonian decomposable. 
\end{theorem}

Of note are two new infinite families of exceptions to Conjecture \ref{conj:main}. Namely,  Theorems \ref{thm:1} and \ref{thm:2} imply that $\vec{C}_n \wr K^*_2$ and  $\vec{C}_n \wr \vec{C}_3$ are not hamiltonian decomposable for all even $n$, respectively. 

\section{Preliminaries}
\label{S:Prm}

Given a digraph $G$, we denote its vertex set as $V(G)$ and its arc set as $A(G)$. In this paper, all digraphs are \textit{strict} meaning that they do not contain loops or repeated arcs. We denote the complete graph on $m$ vertices as $K_m$. The \textit{complete symmetric digraph of order $m$}, denoted $K^*_m$, is the strict digraph on $m$ vertices such that for any two distinct vertices $x$ and $y$, we have $(x, y), (y, x) \in A(K^*_m)$.  The symbol $\vec{C}_m$ denotes the directed $m$-cycle. We refer to the first vertex of a directed path (dipath for short) $P$  as its \textit{source} and denote it as $s(P)$; we also refer to the last vertex of $P$ as its \textit{terminal} and denote it as $t(P)$. Given two dipaths $P=y_1y_2\cdots y_n$ and $Q=x_1 x_2\cdots x_t$, such that $t(P)=s(Q)$, the \textit{concatenation} of $P$ with $Q$ is the directed walk $PQ=y_1y_2\cdots y_{n-1}y_nx_2x_3\cdots x_t$. 

In some of our constructions, we will refer to another product of digraphs. The \textit{tensor product} of digraph $G$ with $H$, denoted $G \times H$, is the digraph on vertex set $V(G) \times V(H)$ such that $((g_1, h_1), (g_2, h_2)) \in A(G \wr H)$ if and only if $(g_1, g_2)\in A(G)$ and $(h_1, h_2)\in A(H)$. 

Next, we establish some key notation and definitions used to describe our constructions.

\begin{notation} \rm
\label{not:conchp5}
Let $H$ be a digraph on $m$ vertices. Let $V(H)=\mathds{Z}_m$ and $V(\vec{C}_n)=\mathds{Z}_n$. Then, $V(\vec{C}_n\wr H)=\{(x, y)\ |\ x\in \mathds{Z}_n, y\in \mathds{Z}_m\}$, and we write shortly $x_y$ for  $(x,y)$. For each $i \in \mathds{Z}_n$, we let $V_i=\{i_0, i_1, \ldots, i_{m-1}\}$. \end{notation}

\begin{definition} \rm
An arc of $\vec{C}_n \wr H$ is of \textit{difference $d$} if it is of the form $(i_{j}, (i+1)_{j+d})$ for some $i \in \mathds{Z}_n$, with addition of the indices done modulo $m$. A \textit{horizontal arc} is an arc of the form $(i_{j_1}, i_{j_2})$ where $(j_1, j_2) \in A(H)$.
 \end{definition}

Definition \ref{def:embed} below allows us to describe a key property of the wreath product. We first note that the subdigraph of $\vec{C}_n \wr H$ induced by the set of vertices $V_i$ is necessarily isomorphic to $H$. We then have the following definition. 

\begin{definition} \rm
\label{def:embed}
Let $H_i$ be the subdigraph of $\vec{C}_n \wr H$ induced by vertex set $V_i$. An \textit{embedding of $H$ into $V_i$} is an isomorphism $\phi: H \mapsto H_i$. 
\end{definition}

It follows from the definition of wreath product that $H$ can be embedded into each $V_i$ arbitrarily. In our constructions, we will embed $H \in \{\vec{C}_m, \vec{P}_m\}$ into each set $V_i$ by specifying the image of $\vec{C}_m$  (or $\vec{P}_m$) by the chosen isomorphism instead of the isomorphism itself.

 Now, we introduce terminology pertaining to cycle decomposition of digraphs. A \textit{decomposition} of  a digraph $G$ is a set $\{H_1, H_2, \ldots, H_r\}$ of pairwise arc-disjoint subdigraphs of $G$ such that $A(G)=A(H_1) \cup A(H_2) \cup \cdots \cup A(H_r)$. A spanning subdigraph of $G$ that is the disjoint union of directed cycles of $G$ is a \textit{directed 2-factor} of $G$. A directed 2-factorization of $G$ is a decomposition of $G$ in which all subdigraphs are 2-factors. Lastly, a directed hamiltonian cycle of $G$ is necessarily a 2-factor of $G$ and a \textit{hamiltonian decomposition} of $G$ is a directed 2-factorization in which all subdigraphs are directed hamiltonian cycles. 
 
Directed 2-factors of $\vec{C}_n \wr H$ that do not contain horizontal arcs can be described as an $n$-tuple of elements from the symmetric group  $S_m$ as follows. 

\begin{notation}\rm
\label{not:Sm}
An $n$-tuple $(\sigma_0, \sigma_1, \ldots, \sigma_{n-1})$ of permutations in $S_m$ corresponds to a directed 2-factor $F$ of $\vec{C}_n \wr H$ with arc set$\{(i_j, (i+1)_{j^{\sigma_i}}) \ |\ j \in \mathds{Z}_m \ \textrm{and} \ i \in \mathds{Z}_n\}$. We then write $F=(\sigma_0, \sigma_1, \ldots, \sigma_{n-1})$.  
\end{notation}

For $\sigma \in S_m$, we denote the number of cycles of $\sigma$ in its disjoint cycle notation as $T(\sigma)$. It is easy to see that the number of directed cycles in $F=(\sigma_0, \sigma_1, \ldots, \sigma_{n-1})$ equals $T(\sigma_0\sigma_1\cdots \sigma_{n-1})$. The directed 2-factor $F=(\sigma_0, \sigma_1, \ldots, \sigma_{n-1})$ is a hamiltonian cycle if and only if $T(\sigma_0\sigma_1\cdots \sigma_{n-1})=1$. 

\section{Tools}
\label{S:tool}

In this section, we survey existing results that are used in our constructions. Firstly, in Section \ref{sec:km}, we will be using decompositions of $K^*_m$ into hamiltonian dipaths to construct hamiltonian decompositions of $\vec{C}_n \wr K^*_m$. The question of existence for these decompositions is settled in \cite{Tilson} by Tillson. When $m$ is even, Tillson \cite{Tilson}  obtains the desired decomposition from a decomposition of $K_m$ into hamiltonian paths constructed in \cite{Tarsi} by simply orienting each path in the decomposition to obtain two hamiltonian dipaths. This gives the following corollary. 
 
\begin{corollary} \cite{Tilson}
\label{cor:dipatheven}
Let $m$ be an even integer. The digraph $K^*_m$ admits a decomposition into hamiltonian dipaths. 
\end{corollary}

As for the case $m$ is odd, Tillson \cite{Tilson} provides original constructions to obtain the following result.

\begin{theorem} \cite{Tilson}
\label{thm:Til}
Let $m$ be an odd integer such that $m \geqslant 7$.  The digraph $K^*_m$ admits a decomposition into hamiltonian dipaths. 
\end{theorem}

To use the results of Corollary \ref{cor:dipatheven} and Theorem \ref{thm:Til}, we need the following property of decompositions of $K^*_m$ into hamiltonian dipaths first established in \cite{Berm2}.

\begin{lemma} \cite{Berm2}
\label{lem:disin}
If $D$ is a decomposition of $K^*_m$ into hamiltonian dipaths, then no two distinct dipaths in $D$ have the same source or the same terminus. 
\end{lemma}
 
Lastly, we will also refer to the following result of Paulraja and Sivansankar regarding hamiltonian decompositions of $\vec{C}_{n}\times K^*_m$. 
 
\begin{theorem} \cite{Praja2}
\label{thm:cat} 
Let $n\geqslant 4$ be an even integer, $m \geqslant 3$, and $m \neq 4$. The digraph $\vec{C}_{n}\times K^*_m$ is hamiltonian decomposable. 
\end{theorem}

\section{Decompositions of $\vec{C}_n \wr K^*_m$ into directed hamiltonian cycles}
\label{sec:km}

In this section, we show that $\vec{C}_n \wr K^*_m$ is hamiltonian decomposable for all pairs of positive integers $(n, m)$, where $n$ is even, with the exception of $(n,m)=(2,3)$ and all pairs in which $m=2$. If $(n,m)=(2,3)$, then $\vec{C}_2 \wr K^*_3 \cong K^*_6$. In \cite{Berm2}, Bermond and Faber have shown that $K^*_6$ is not hamiltonian decomposable. As for the case $m=2$, non-existence of a hamiltonian decomposition of $\vec{C}_n \wr K^*_m$  is established in Proposition \ref{prop:notm2} below. 

\begin{proposition}
\label{prop:notm2}
Let $n$ be an even integer. The digraph $\vec{C}_n \wr K^*_2$ is not hamiltonian decomposable. 
\end{proposition}

\begin{proof} Suppose that $\mathcal{D}=\{C^0, C^1, C^2\}$ is a decomposition of  $\vec{C}_n \wr K^*_2$ into directed hamiltonian cycles. 

We first show that all horizontal arcs are contained in just two directed hamiltonian cycles of $\mathcal{D}$. Without loss of generality, suppose that $C^0$ contains the horizontal arc $(0_0, 0_1)$ and that the first vertex of  $C^0$ is $0_0$. We claim that $C^0$ contains $n$ horizontal arcs. Suppose that there exists some $i \in \mathds{Z}_n$, $i \neq 0$, such that $C^0$ contains neither arc $(i_0, i_1)$ nor $(i_1, i_0)$. Without loss of generality, suppose that $i_0$ appears before $i_1$ in $C^0$. Note that the sequence of vertices of $C^0$ is not decreasing in the first coordinate. The subdipath of $C^0$ with source $i_0$ and terminal $i_1$ does not contain the arc $(i_0, i_1)$, and hence contains at least one of $0_0$ and $0_1$. This implies that $C^0$ has a repeated vertex ($0_0$ or $0_1$) other than its endpoints, a contradiction. Therefore, the directed cycle $C^0$ contains at least $n$ horizontal arcs, and thus exactly $n$ horizontal arcs of $\vec{C}_n \wr K^*_2$. By a similar reasoning, we can show that $C^1$ contains the $n$ horizontal arcs that do not appear in $C^0$. 

It follows from the above property that $S=A(C^2)$ contains an even number of arcs of difference 1. In addition, we see that $(i_0, (i+1)_1) \in S$ if and only if $(i_1, (i+1)_0) \in S$. Similarly, we see that $(i_1, (i+1)_1) \in S$ if and only if $(i_0, (i+1)_0) \in S$. Consequently, the digraph $C^2$ is actually the disjoint union of two directed cycles of length $n$, a contradiction. \end{proof}

Next, we show that $\vec{C}_n \wr K^*_m$ is hamiltonian decomposable for all even $n$ and all $m \geqslant 3$. We address this problem in two stages. First, we construct a decomposition of $\vec{C}_n \wr K^*_m$ into directed hamiltonian cycles for all $m\geqslant 6$. The method used in the proof of Proposition \ref{prop:complete2} does not apply to the case $m \in \{3,4,5\}$. These three exceptions are addressed separately in Lemma \ref{lem:spec35}.

\begin{proposition}
\label{prop:complete2}
Let $n$ be an even integer and $m\geqslant 6$. The digraph $\vec{C}_n \wr K^*_m$ is hamiltonian decomposable. 
\end{proposition}

\begin{proof} First, we consider the case $n=2$. In that case, we see that $\vec{C}_2 \wr K^*_m\cong K^*_{2m}$. Since $m \geqslant 6$, the digraph $K^*_{2m}$ is hamiltonian decomposable. 

Next, we consider the case $n \geqslant 4$. By Corollary \ref{cor:dipatheven} and Theorem \ref{thm:Til} there exists a decomposition $D=\{P_1, P_2, \ldots, P_m\}$  of $K^*_m$ into hamiltonian dipaths. For each $P_j \in D$, assume that $P_j=v_1^j\, v_2^j \cdots v_m^j$. Using $D$, we can obtain a second decomposition of $K^*_m$ into directed hamiltonian dipaths, $D'=\{P'_1, P'_2, \ldots, P'_m\}$, where $P'_j=v_m^j\, v_{m-1}^j\cdots v_1^j$. 

For each $P_j \in D$, we then construct the following directed cycle of $\vec{C}_n \wr K^*_m$:

\begin{center}
$C^j=0_{v^j_1}\, 0_{v^j_2}\cdots 0_{v^j_m}\, 1_{v^j_m}\, 1_{v^j_{m-1}}\, 1_{v^j_{m-2}}\cdots 1_{v^j_1}\, 2_{v^j_1}\, 2_{v^j_2}\, 2_{v^j_3}\cdots\, (n-1)_{v^j_2}\, (n-1)_{v^j_1} \,  0_{v^j_1}$.
\end{center}

\noindent The directed cycle  $C^j$ is in fact a directed hamiltonian cycle of $\vec{C}_n \wr K^*_m$. The directed cycle $C^j$ contains the following arcs of difference 0: 

\begin{center}
$\{(0_{v^j_m}, 1_{v^j_m}), (1_{v^j_1}, 2_{v^j_1}), (2_{v^j_m}, 3_{v^j_m}), (3_{v^j_1}, 4_{v^j_1}), \ldots, ((n-1)_{v^j_1}, 0_{v^j_1})\}$. 
\end{center}

\noindent By Lemma \ref{lem:disin}, $v^j_1$ is the source of exactly one dipath in $D$, and $v^j_m$ is the terminal of exactly one dipath in $D$; likewise for $D'$ in which $v^j_m$ and $v^j_1$ will appear as the source and terminal, respectively, of exactly one dipath in $D'$. It follows that each arc of difference 0 appears exactly once in $F=\{C^1, C^2, \ldots, C^m\}$.  Furthermore, since $D$ and $D'$ are decomposition of $K^*_m$ into hamiltonian dipaths, each horizontal arc of $\vec{C}_n \wr K^*_m$ also appears exactly once in $F$. Otherwise, no arc of difference $d>0$ appears in $F$. As a result, we have that $\{C^1, C^2, \ldots, C^m, \vec{C}_n \times K^*_m\}$ is a decomposition of $\vec{C}_n \wr K^*_m$. By Theorem \ref{thm:cat}, the digraph $\vec{C}_n \times K^*_m$ admits a decomposition $F$ into directed hamiltonian cycles when $m \geqslant 6$. Since $\vec{C}_n \times K^*_m$ is a spanning subdigraph of $\vec{C}_n \wr K^*_m$, each of the $m-1$ directed hamiltonian cycles in $F$ is also directed hamiltonian cycles of $\vec{C}_n \wr K^*_m$. The set $\{C^1, C^2, \ldots, C^m\} \cup F$ is a decomposition of $\vec{C}_n \wr K^*_m$ into directed hamiltonian cycles. \end{proof}

Three small cases are omitted in the statement of Proposition \ref{prop:complete2}, specifically the cases in which $m \in \{3,4,5\}$. The construction given in the proof of Proposition \ref{prop:complete2} does not work for these small cases for one of two reasons. If $m \in \{3,5\}$,  the digraph $K^*_m$ does not admit a decomposition into hamiltonian dipaths \cite{Berm2}.  If $m=4$, Theorem \ref{thm:cat} does not apply to the digraph $\vec{C}_n \times K^*_m$. Note that $K^*_4$ is not hamiltonian decomposable, as shown in \cite{Berm2}, and thus needs not be considered with regards to Conjecture \ref{conj:main}. However, we address this special case for completeness of results. 

\begin{lemma}
\label{lem:spec35}
Let $n$ be an even integer and $m \in \{3,4,5\}$. The digraph $\vec{C}_n \wr K^*_m$ is hamiltonian decomposable if and only if $(n, m) \neq (2,3)$. 
\end{lemma}

\begin{proof} We consider three cases, one for each $m\in \{3,4,5\}$.

\noindent \underline{Case 1}: $m=3$. Since $\vec{C}_2\wr K^*_3=K^*_6$,  we see that $\vec{C}_2\wr K^*_3$ is not hamiltonian decomposable. 

Conversely, let $n\geqslant 4$. We begin by constructing 9 dipaths of $\vec{C}_n \wr K^*_3$ as follows: 

\smallskip
{\centering
  $ \displaystyle
    \begin{aligned} 
&X_0=0_2\, 0_0\, 0_1\, 1_0\,1_1\, 1_2\, 2_2\, 2_1\, 2_0\, 3_0\, 3_2\, 3_1\,4_2; &&X_3=0_1\, 0_2\, 1_2\,2_1\,3_0\, 4_0;\, &&X_6=0_2\, 1_1\, 2_2\, 3_0\,3_1\,4_1;\\
&X_1=0_1\, 1_2\, 1_1\,2_1\, 2_2\, 3_1\, 4_0;&& X_4=0_0\, 1_1\, 1_0\,2_2\,2_0\,3_1\, 3_2\, 4_1;&&X_7=0_2\, 0_1\, 1_1\,2_0\,2_2\,3_2\, 4_0;\\
&X_2=0_0\, 0_2\, 1_0\,2_0\,3_2\, 3_0\,4_1; &&X_5=0_1\, 0_0\, 1_0\,1_2\,2_0\,2_1\,3_2\,4_2;  &&X_8=0_0\, 1_2\, 1_0\,2_1\,3_1\,3_0\,4_2.\\
     \end{aligned}
  $ 
\par}
\smallskip 

 \noindent Refer to Figure \ref{fig:Xi} for an illustration of the 9 dipaths above. 

 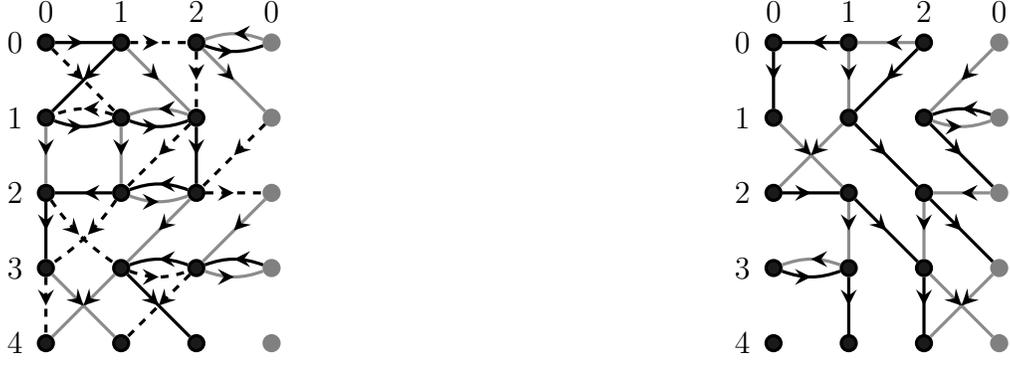
\begin{figure}[htpb!]
 \begin{center}
 \begin{subfigure}[c]{0.45 \textwidth}
 \begin{center}
 \begin{tikzpicture}[
  very thick,
  every node/.style={circle,draw=black,fill=black!90}, inner sep=2]

  \node (x0) at (0.0,7.0) [label=above:$0$] [label=left:$0$] {};
  \node (x1) at (1.0,7.0)[label=above:$1$]  {};
  \node (x2) at (2.0,7.0) [label=above:$2$] {};
  \node (x3) at (3.0,7.0) [draw=gray, fill=gray, label=above:$0$] {};
  \node (y0) at (0.0,6.0) [label=left:$1$]  {};
  \node (y1) at (1.0,6.0) {};
  \node (y2) at (2.0,6.0) {};
  \node (y3) at (3.0,6.0) [draw=gray, fill=gray] {};
   \node (z0) at (0.0,5.0) [label=left:$2$]  {};
  \node (z1) at (1.0,5.0) {};
  \node (z2) at (2.0,5.0) {};
  \node (z3) at (3.0,5.0) [draw=gray, fill=gray] {};
  \node (a0) at (0.0,4.0) [label=left:$3$]  {};
  \node (a1) at (1.0,4.0) {};
  \node (a2) at (2.0,4.0) {};
  \node (a3) at (3.0,4.0) [draw=gray, fill=gray] {};
  \node (b0) at (0.0,3.0) [label=left:$4$]  {};
  \node (b1) at (1.0,3.0) {};
  \node (b2) at (2.0,3.0) {};
  \node (b3) at (3.0,3.0) [draw=gray, fill=gray] {};

\path [very thick, draw=color1, postaction={very thick, on each segment={mid arrow}}]
	(x0) to (x1)
	(x2) to [bend right=20, near start]   (x3)
	(x1) to  (y0)

	(y0) to [bend right=20] (y1)
	(y1) to [bend right=20] (y2)
	(y2) to (z2)
	
	(z1) to (z0)
	(z2) to [bend right=20] (z1)
	(z0) to (a0)
	(a3) to [bend right=20] (a2)
	(a2) to  [bend right=20]  (a1)
	(a1) to (b2)
	;
\path [very thick, draw=color2, postaction={very thick, on each segment={mid arrow}}]
	(x3) to [bend right=20]  (x2)
	(x2) to (y3)
	(y0) to  (z0)

	(z3) to  (a2)
	(a2) to [bend right=20] (a3)
	(a0) to (b1)
	
	(x1) to (y2)
	(y2) to [bend right=20] (y1)
	(y1) to (z1)
	(z1) to [bend right=20] (z2)
	(z2) to (a1)
	(a1) to (b0)
	;
 \path [very thick, dashed, draw=color1, postaction={very thick, on each segment={mid arrow}}]
	(x1) to (x2)
	(x2) to (y2)
	
	(y2) to  (z1)
	(z1) to [bend left=10] (a0)
	
	(a0) to  (b0)
	(x0) to (y1)
	
	(y1) to [bend right=20]    (y0)
	(y3) to (z2)
	
	(z2) to (z3)
	
	(z0) to [bend right=10, near end] (a1)
	(a1) to  [bend right=20] (a2)
	(a2) to (b1);
	
\end{tikzpicture}
\end{center}
\caption{Illustration of  $X_0$ (black), $X_1$ and $X_2$ (grey), and  $X_3$ and $X_4$ (dashed).}
\end{subfigure}
\hfill
  \begin{subfigure}[c]{0.45 \textwidth}
 \begin{center}
 \begin{tikzpicture}[
  very thick,
  every node/.style={circle,draw=black,fill=black!90}, inner sep=2]

  \node (x0) at (0.0,7.0) [label=above:$0$] [label=left:$0$] {};
  \node (x1) at (1.0,7.0)[label=above:$1$]  {};
  \node (x2) at (2.0,7.0) [label=above:$2$] {};
  \node (x3) at (3.0,7.0) [draw=gray, fill=gray, label=above:$0$] {};
  \node (y0) at (0.0,6.0) [label=left:$1$]  {};
  \node (y1) at (1.0,6.0) {};
  \node (y2) at (2.0,6.0) {};
  \node (y3) at (3.0,6.0) [draw=gray, fill=gray] {};
   \node (z0) at (0.0,5.0) [label=left:$2$]  {};
  \node (z1) at (1.0,5.0) {};
  \node (z2) at (2.0,5.0) {};
  \node (z3) at (3.0,5.0) [draw=gray, fill=gray] {};
  \node (a0) at (0.0,4.0) [label=left:$3$]  {};
  \node (a1) at (1.0,4.0) {};
  \node (a2) at (2.0,4.0) {};
  \node (a3) at (3.0,4.0) [draw=gray, fill=gray] {};
  \node (b0) at (0.0,3.0) [label=left:$4$]  {};
  \node (b1) at (1.0,3.0) {};
  \node (b2) at (2.0,3.0) {};
  \node (b3) at (3.0,3.0) [draw=gray, fill=gray] {};

\path [very thick, draw=color1, postaction={very thick, on each segment={mid arrow}}]
	(x1) to (x0)
	(x0) to (y0)
	(y3) to [bend right=20]   (y2)

	(y2) to (z3)
	(z0) to (z1)
	(z1) to (a2)
	
	(a2) to (b2)
	(x2) to (y1)
	(y1) to (z2)
	(z2) to (a3)
	(a0) to [bend right=20] (a1)
	(a1) to (b1)
	;
\path [very thick, draw=color2, postaction={very thick, on each segment={mid arrow}}]
	(x2) to (x1)
	(x1) to (y1)
	(y1) to  (z0)

	(z3) to  (z2)
	(z2) to (a2)
	(a2) to (b3)
	
	(x3) to (y2)
	(y2) to [bend right=20] (y3)
	(y0) to (z1)
	(z1) to (a1)
	(a1) to [bend right=20]  (a0)
	(a3) to (b2);
\end{tikzpicture}
\end{center}
\caption{Illustration of  $X_5$ and $X_6$ (black), and $X_7$ and $X_8$ (grey).}
\end{subfigure}
\end{center}
\caption{Key dipaths in the construction of a decomposition of $\vec{C}_n \wr K^*_3$ into directed hamiltonian cycles.}
\label{fig:Xi}
\end{figure}

\noindent {SUBCASE 1.1:} $n=4$. Then $4_2=0_2$, meaning that $C^0=X_0$ is a directed cycle of length $12$. Furthermore, we obtain the following four directed cycles of length 12 by concatenating the dipaths above as follows:  $C^1=X_1X_2$, $C^2=X_3X_4$, $C^3=X_5X_6$, and $C^4=X_7X_8$. The set $D=\{C^0, C^1, C^2, C^3, C^4\}$ is the a directed hamiltonian decomposition of $\vec{C}_4 \wr K^*_3$. 
 
\noindent {SUBCASE 1.2:}  $n \geqslant 6$. Let $\mathds{I}= \{4,6, \ldots, n-2\}$. For each $i \in \mathds{I}$, we construct the following set of 9 dipaths:

 \begin{multicols}{2}
{\centering
  $ \displaystyle
    \begin{aligned} 
&M^i=i_2\, i_0\, i_1\,(i+1)_2\, (i+1)_1\, (i+1)_0\, (i+2)_2 ;\\
&N^i_0=i_0\, i_2\, (i+1)_1\, (i+2)_0;\\
&N^i_1 = i_1\, (i+1)_0\, (i+1)_2\, (i+2)_1;\\
&O^i_0=i_0\, (i+1)_1\, (i+1)_2\, (i+2)_0;\\
&O^i_1=i_1\, i_2\, (i+1)_0\, (i+2)_1\,;\\
     \end{aligned}
  $ 
\par}

{\centering
  $ \displaystyle
    \begin{aligned} 
&P^i_0=i_2\, (i+1)_2\, (i+2)_2 ;\\
&P^i_1=i_1\, i_0\, (i+1)_0\,(i+1)_1\,(i+2)_1  ;\\
& Q^i_0=i_0\, (i+1)_2\, (i+1)_0\, (i+2)_0;\\
&Q^i_1=i_2\, i_1\, (i+1)_1\, (i+2)_2.\\
     \end{aligned}
  $ 
  \par}
  \end{multicols}

\noindent \noindent Refer to Figure \ref{fig:ext} for an illustration of these 9 dipaths when $i=4$. It is straightforward to verify that, for each $i \in  \mathds{I}$, dipaths in $\{M^i, N^i_0, N^i_1, \ldots, Q^i_1\}$ are pairwise arc-disjoint.

 \begin{figure}[htpb!]
 \begin{center}
 \begin{subfigure}[c]{0.45 \textwidth}
 \begin{center}
 \begin{tikzpicture}[
  very thick,
  every node/.style={circle,draw=black,fill=black!90}, inner sep=2]

  \node (x0) at (0.0,7.0) [label=above:$0$] [label=left:$4$] {};
  \node (x1) at (1.0,7.0)[label=above:$1$]  {};
  \node (x2) at (2.0,7.0) [label=above:$2$] {};
  \node (x3) at (3.0,7.0) [draw=gray, fill=gray, label=above:$0$] {};
  \node (y0) at (0.0,6.0) [label=left:$5$]  {};
  \node (y1) at (1.0,6.0) {};
  \node (y2) at (2.0,6.0) {};
  \node (y3) at (3.0,6.0) [draw=gray, fill=gray] {};
   \node (z0) at (0.0,5.0) [label=left:$6$]  {};
  \node (z1) at (1.0,5.0) {};
  \node (z2) at (2.0,5.0) {};
  \node (z3) at (3.0,5.0) [draw=gray, fill=gray] {};

\path [very thick, draw=color2, postaction={very thick, on each segment={mid arrow}}]
	(x1) to (y0)
	(y3) to (y2)
	(y2) to  (z1)

	(x3) to [bend right=20] (x2)
	(x2) to (y1)
	(y1) to (z0);
\path [very thick, draw=color1, postaction={very thick, on each segment={mid arrow}}]
	(x2) to [bend right=20]   (x3)
	(x0) to  (x1)
	(x1) to  (y2)

	(y2) to [bend right=20] (y1)
	(y1) to (y0)
	(y3) to (z2);
	;
 \path [very thick, dashed, draw=color1, postaction={very thick, on each segment={mid arrow}}]
	(x1) to (x2)
	(x2) to (y3)
	
	(y0) to  (z1)
	(x0) to (y1)
	
	(y1) to [bend right=20]  (y2)
	(y2) to (z3);
	
\end{tikzpicture}
\end{center}
\caption{Illustration of  $M^4$ (black), $N_0^4$ and $N_1^4$ (dashed), and  $O^4_0$ and $O^4_1$ (grey).}
\end{subfigure}
\hfill
  \begin{subfigure}[c]{0.45 \textwidth}
 \begin{center}
 \begin{tikzpicture}[
  very thick,
  every node/.style={circle,draw=black,fill=black!90}, inner sep=2]

  \node (x0) at (0.0,7.0) [label=above:$0$] [label=left:$4$] {};
  \node (x1) at (1.0,7.0)[label=above:$1$]  {};
  \node (x2) at (2.0,7.0) [label=above:$2$] {};
  \node (x3) at (3.0,7.0) [draw=gray, fill=gray, label=above:$0$] {};
  \node (y0) at (0.0,6.0) [label=left:$5$]  {};
  \node (y1) at (1.0,6.0) {};
  \node (y2) at (2.0,6.0) {};
  \node (y3) at (3.0,6.0) [draw=gray, fill=gray] {};
     \node (z0) at (0.0,5.0) [label=left:$6$]  {};
  \node (z1) at (1.0,5.0) {};
  \node (z2) at (2.0,5.0) {};
  \node (z3) at (3.0,5.0) [draw=gray, fill=gray] {};

\path [very thick, draw=color1, postaction={very thick, on each segment={mid arrow}}]
	(x1) to (x0)
	(x0) to (y0)
	(y0) to (y1)

	(y1) to (z1)
	(x2) to (y2)
	(y2) to (z2)
	;
\path [very thick, draw=color2, postaction={very thick, on each segment={mid arrow}}]
	(x2) to (x1)
	(x1) to (y1)
	(y1) to  (z2)

	(x3) to  (y2)
	(y2) to (y3)
	(y3) to (z3);
\end{tikzpicture}
\end{center}
\caption{Illustration of  $P_0^4$ and $P_1^4$ (black), and $Q_0^4$ and $Q_1^4$ (grey).}
\end{subfigure}
\end{center}
\caption{Key dipaths in the construction of a decomposition of $\vec{C}_n \wr K^*_3$ into directed hamiltonian cycles for $n \geqslant 6$.}
\label{fig:ext}
\end{figure}

We now use these dipaths to construct five directed hamiltonian cycles of $\vec{C}_n \wr K^*_3$. First, observe that each $M^i$ is a dipath of length 6. In addition, note that $t(M^i)=s(M^{i+2})$ and that $M^i$ and $M^{i+2}$ have no other vertices in common. If $|i-j|>2$, then dipaths $M^i$ and $M^j$ are vertex-disjoint. Lastly, we see that $s(M^4)=t(X_0)=4_2$, $t(M^{n-2})=s(X_0)=0_2$, and that $X_0$ has no other vertices in common with dipaths in $\{M^4, M^6, \ldots, M^{n-2}\}$. This means that we can construct the following directed cycle of length $3n$: $C^0=X_0M^4M^6M^8\cdots M^{n-2}$. 

By applying a similar reasoning, we can construct the following four directed cycles of length $3n$: 
 
\smallskip
{\centering
  $ \displaystyle
    \begin{aligned} 
&C^1=X_1N_0^4N_0^6\cdots N_0^{n-2}X_2N_1^4N_1^6\cdots N_1^{n-2};&&C^3=X_5P^4_0P^6_0\cdots P^{n-2}_0X_6P^4_1P^6_1\cdots P^{n-2}_1;\\
&C^2=X_3O^4_0O^6_0\cdots O^{n-2}_0X_4O^4_1O^6_1\cdots O^{n-2}_1;&&C^4=X_7Q^4_0Q^6_0\cdots Q^{n-2}_0X_8Q^4_1Q^6_1\cdots Q^{n-2}_1.
     \end{aligned}
  $ 
\par}
\smallskip 

\noindent  Observe that the dipaths in  $\cup_{i=4}^{n-2}\{M^i, N^i_0, N^i_1, \ldots, Q^i_1\}$ are pairwise arc-disjoint. Since the dipaths in $\{X_0, X_1, \ldots, X_8\}$ are also pairwise arc-disjoint, this implies that the directed cycles in $D=\{C^0, C^1, C^2, C^3, C^4\}$ are pairwise arc-disjoint. Consequently, the set $D$ is a directed hamiltonian decomposition of $\vec{C}_n \wr K^*_3$. 

\noindent \underline{Case 2}: $m=4$. Observe that, if $n=2$, then $\vec{C}_n\wr K^*_4=K^*_8$. The digraph $K^*_8$ is hamiltonian decomposable \cite{Tilson}. 

We now consider the case $n \geqslant 4$. We begin by constructing three directed cycles of length $4n$ as follows:

\smallskip
{\centering
  $ \displaystyle
    \begin{aligned} 
&C^0=0_3\,0_2\,0_1\,0_0\, 1_0\, 1_1\, 1_2\,1_3\, 2_3\, 2_2\, 2_1 \cdots (n-1)_0\, (n-1)_1\,(n-1)_2\,(n-1)_3\, 0_3;\\
&C^1=0_1\,0_3\,0_0\,0_2\, 1_2\, 1_0\, 1_3\,1_1\, 2_1\, 2_3\, 2_0 \cdots (n-1)_2\, (n-1)_0\,(n-1)_3\,(n-1)_1\,0_1;\\
&C^2=0_2\,0_0\,0_3\,0_1\, 1_1\, 1_3\, 1_0\,1_2\, 2_2\, 2_0\, 2_3 \cdots (n-1)_1\, (n-1)_3\,(n-1)_0\,(n-1)_2\,0_2.\\
    \end{aligned}
  $ 
\par}
\smallskip

\noindent Observe that $C^0$, $C^1$, and $C^2$ are pairwise arc-disjoint. Let $S=A(C^0)\cup A(C^1)\cup A(C^2)$. All arcs of difference 0 are contained in $S$ except for arcs in the set  $\{(i_3, (i+1)_3), ((i+1)_0, (i+2)_0) \ | \ i \in \mathds{Z}_n\ \textrm{and $i$ even}\}$. 

\noindent Furthermore, the set $S$ also contains all horizontal arcs except for arcs in the set 

\smallskip
{\centering
  $ \displaystyle
    \begin{aligned} 
&\{(i_0, i_1), (i_1, i_2), (i_2, i_3), ((i+1)_3, (i+1)_2), ((i+1)_2, (i+1)_1), ((i+1)_1, (i+1)_0) \ | \ i \in \mathds{Z}_n,\ \textrm{$i$ even} \}.
    \end{aligned}
  $ 
\par}
\smallskip
\noindent Lastly, note that $S$ does not contain any arc of difference $d>0$.

We now construct four directed hamiltonian cycles by first constructing the following 13 dipaths:

\vspace{-2mm}
\begin{multicols}{4}
{\centering
  $ \displaystyle
    \begin{aligned} 
&U_0=0_1\, 0_2\, 0_3\, 1_2\, 1_1\, 1_0\,2_0\,;\\
&U_1=0_0\, 1_3\, 2_1; \\
&W_0=0_0\, 0_1\, 1_3\,1_2\,2_3;\\
& W_1=0_3\, 1_0\, 2_2; 
    \end{aligned}
  $ 
\par}
{\centering
  $ \displaystyle
    \begin{aligned} 
    &W_2=0_2\, 1_1\, 2_0;\\ 
    &X_0=0_0\,1_1\,2_3;\\
    &X_1=0_3\, 1_3\, 2_2; \\   
    &X_2=0_2\, 1_0\, 2_1;\\ 
    \end{aligned}
  $ 
\par}
{\centering
  $ \displaystyle
    \begin{aligned} 
    &X_3=0_1\, 1_2\, 2_0; \\    
    &Y_0=0_0\,1_2\,2_1; \\
    &Y_1=0_1\, 1_0\, 2_3; \\
    &Y_2=0_3\, 1_1\, 2_2
  \end{aligned}
  $ 
\par}

{\centering
  $ \displaystyle
    \begin{aligned} 
&Y_3=0_2\, 1_3\, 2_0.
  \end{aligned}
  $ 
\par}
\end{multicols}

Let $\mathds{I}=\{2,4, 6, \ldots, n-2\}$. For each $i\in \mathds{I}$, we construct the following set of 13 dipaths:

\vspace{-2mm}
\begin{multicols}{3}
{\centering
  $ \displaystyle
    \begin{aligned} 
&M^i_0=i_0\, (i+1)_3\, (i+2)_0;\\  
&M^i_1=i_1\, i_2\, i_3\, (i+1)_2\,(i+1)_1\, (i+1)_0\, (i+2)_1;\,\\
&N^i_0=i_3\, (i+1)_3\, (i+1)_2\, (i+2)_3;\\
&N^i_1=i_2\, (i+1)_1\, (i+2)_2;\\  
&N^i_2=i_0\, i_1\, (i+1)_0\,(i+2)_0;\\
    \end{aligned}
  $ 
\par}
\hspace{1.3cm}
{
  $ \displaystyle
    \begin{aligned} 
&O^i_0=i_3\, (i+1)_0\, (i+2)_3;  \\
&O^i_1=i_2\, (i+1)_3\, (i+2)_2;\, \\
&O^i_2=i_1\, (i+1)_2\, (i+2)_1;   \\
&O^i_3=i_0\, (i+1)_1\, (i+2)_0;\,\\
&P^i_0=i_1\, (i+1)_3\, (i+2)_1; \\
    \end{aligned}
  $ 
\par}
\hspace{-0.5cm}
{\centering
  $ \displaystyle
    \begin{aligned}
&P^i_1=i_3\, (i+1)_1\, (i+2)_3 ;\\
&P^i_2=i_2\, (i+1)_0\, (i+2)_2;  \\
&P^i_3=i_0\, (i+1)_2\, (i+2)_0.\\
    \end{aligned}
  $ 
\par}
\end{multicols}

\noindent We then construct the following four directed cycles of length $4n$:

\smallskip
{\centering
  $ \displaystyle
    \begin{aligned} 
&C^4=U_0M^2_0M^4_0\cdots M^{n-2}_0 U_1 M^2_1M^4_1\cdots M^{n-2}_1;\\    
&C^5=W_0N^2_0N^4_0\cdots N^{n-2}_0 W_1 N^2_1N^4_1\cdots N^{n-2}_1W_2N^2_2N^4_2\cdots N^{n-2}_2;\\
&C^6=X_0O^2_0O^4_0\cdots O^{n-2}_0 X_1 O^2_1O^4_1\cdots O^{n-2}_1X_2O^2_2O^4_2\cdots O^{n-2}_2X_3O^2_3O^4_3\cdots O^{n-2}_3;\\
&C^7=Y_0P^2_0P^4_0\cdots P^{n-2}_0 Y_1 P^2_1P^4_1\cdots P^{n-2}_1Y_2P^2_2P^4_2\cdots P^{n-2}_2Y_3P^2_3P^4_3\cdots P^{n-2}_3.\\
    \end{aligned}
  $ 
\par}
\smallskip

\noindent We point out that $S'=A(C^3)\cup A(C^4)\cup A(C^5)\cup A(C^6)$ contains all horizontal arcs and all arcs of difference 0 that do not appear in $S$. Furthermore, each arc of difference $d>0$ appears exactly once in $\{C^3, C^4, C^5, C^6\}$. This means that $D=\{C^0, C^1, \ldots, C^6\}$ is a directed hamiltonian decomposition of $\vec{C}_n \wr K^*_4$. 

\noindent \underline{Case 3}: $m=5$. Observe that, if $n=2$, then $\vec{C}_n\wr K^*_5=K^*_{10}$. The digraph $K^*_{10}$ is hamiltonian decomposable \cite{Tilson}. 

We now construct a decomposition of $\vec{C}_n \wr K^*_5$ into 9 directed hamiltonian cycles for all even $n \geqslant 4$. First, we construct the following set of 25 dipaths:

\vspace{-3mm}
    \begin{multicols}{2}
{\centering
  $ \displaystyle
    \begin{aligned} 
&L_0=0_0\, 0_1\, 0_2\, 0_3\, 0_4\, 1_3\, 1_2\,1_1\, 1_0\, 1_4\, 2_0;\\
&L_1=0_4\, 0_3\, 0_2\, 0_1\, 0_0\, 1_4\, 1_0\,1_1\, 1_2\, 1_3\, 2_4;\\
&L_2=0_1\,0_3\,0_0\, 0_2\, 0_4\, 1_4\, 1_2\, 1_0\, 1_3\,1_1\, 2_1\,;\\
&L_3=0_3\, 0_1\, 0_4\, 0_2\, 0_0\, 1_0\, 1_2\,1_4\, 1_1\, 1_3\, 2_3;\\
\end{aligned}
  $ 
\par}
\vspace{0.5 ex}
{\centering
  $ \displaystyle
    \begin{aligned} 
&~U_0=0_4\, 0_0\, 1_1\,2_2; &&U_1= 0_2\, 1_3\, 1_4\, 2_3;\\
&~U_2=0_3\, 1_2\, 2_1; &&U_3=0_1\, 1_0\, 2_4;\\  
&~W_0=0_0\, 0_4\, 1_0\,2_1;&& W_1=0_1\, 1_2\, 2_3;\\
&~W_2=0_3\, 1_4\, 1_3\,2_2 ;&& W_3=0_2\, 1_1\, 2_0; \\
    \end{aligned}
  $ 
\par}

{\centering
  $ \displaystyle
    \begin{aligned} 
&X_0=0_4\, 0_1\, 1_1\,2_3; &&X_1=0_3\, 1_3\, 1_0\, 2_0;\\
&X_2=0_0\, 1_2\, 2_2\,; &&X_3=0_2\, 1_4\, 2_4; \\ 
&Y_0=0_0\, 0_3\, 1_0\,2_2; &&Y_1=0_2\, 1_2\, 2_4;\\
&Y_2=0_4\, 1_1\, 1_4\, 2_1;&& Y_3=0_1\, 1_3\, 2_0; \\
&Z_0=0_0\, 1_3\, 2_1\, ; && Z_1=0_1\, 1_4\, 2_2;\\ 
&Z_2=0_2\, 1_0\, 2_3 ;&& Z_3=0_3\, 1_1\, 2_4; \\ 
&Z_4=0_4\, 1_2\,2_0.&&\\
    \end{aligned}
  $ 
\par}

\end{multicols}
\vspace{-3mm}

\noindent Next, let $\mathds{I}=\{2,4, \ldots, n-2\}$. For each $i \in \mathds{I}$, we form the following set of four dipaths:

{\centering
  $ \displaystyle
    \begin{aligned} 
&M^i_0=i_0\, i_1\, i_2\, i_3\, i_4\, (i+1)_3\, (i+1)_2\,(i+1)_1\, (i+1)_0\, (i+1)_4\, (i+2)_0;\\
&M^i_1=i_4\, i_3\, i_2\, i_1\, i_0\, (i+1)_1\, (i+1)_2\,(i+1)_3\, (i+1)_4\, (i+1)_0\, (i+2)_4;\\
&M^i_2=i_1\,i_3\,i_0\, i_2\, i_4\, (i+1)_4\, (i+1)_2\, (i+1)_0\, (i+1)_3\,(i+1)_1\, (i+2)_1;\\
&M^i_3=i_3\, i_1\, i_4\, i_2\, i_0\, (i+1)_0\, (i+1)_2\,(i+1)_4\, (i+1)_1\, (i+1)_3\, (i+2)_3.\\
    \end{aligned}
  $ 
\par}

\noindent We then form the following four directed cycles of length $5n$:
 
\smallskip
{\centering
  $ \displaystyle
    \begin{aligned} 
&C^0=L_0M^2_0M^4_0M^6_0\cdots M^{n-2}_0; &&C^2=L_2M^2_2M^4_2M^6_2\cdots M^{n-2}_2;\\ 
& C^1=L_1M^2_1M^4_1M^6_1\cdots M^{n-2}_1; &&C^3=L_3M^3_1M^4_3M^6_3\cdots M^{n-2}_3. \\
    \end{aligned}
  $ 
\par}
\smallskip

\noindent Next, for each $i \in \mathds{I}$, we construct a set of 21 dipaths as follows:

\begin{multicols}{3}
{\centering
  $ \displaystyle
    \begin{aligned} 
 &N^i_0=i_2\, (i+1)_1\, (i+2)_2;\\
 &N^i_1=i_3\, (i+1)_2\, (i+2)_3; \\
& N^i_2=i_1\, (i+1)_0\, (i+2)_1; \\
&N^i_3=i_4\, i_0 (i+1)_4(i+1)_3(i+2)_4;\\
 & O^i_0=i_1\, (i+1)_2\, (i+2)_1; \\
&O^i_1=i_3\, (i+1)_4\, (i+2)_3;\\
& O^i_2=i_2\, (i+1)_3\, (i+2)_2; \\
     \end{aligned}
  $ 
\par}
\hspace{-7mm}
{
  $ \displaystyle
    \begin{aligned} 
  &O^i_3=i_0\, i_4 (i+1)_0 (i+1)_1 (i+2)_0;\,\\
& P^i_0=i_3 (i+1)_3 (i+1)_0(i+2)_3; \\
&P^i_1=i_0\, (i+1)_2\, (i+2)_0;\\
    & P^i_2=i_2\, (i+1)_4\, (i+2)_2; \\
       &P^i_3=i_4\, i_1\, (i+1)_1\,(i+2)_4,\\
    &Q^i_0=i_2\, (i+1)_2\, (i+2)_2; \\
&Q^i_1=i_4\, (i+1)_1\, (i+1)_4\, (i+2)_4 ;\\
    \end{aligned}
  $ 
\par}
\hspace{0.2cm}
{
  $ \displaystyle
    \begin{aligned} 
     & Q^i_2=i_1\, (i+1)_3\, (i+2)_1; \\
      &Q^i_3=i_0\, i_3\, (i+1)_0\,(i+2)_0;  \\
  & R^i_0=i_1\, (i+1)_4\, (i+2)_1; \\
&R^i_1=i_2\, (i+1)_0\, (i+2)_2 ; \\
& R^i_2=i_3\, (i+1)_1\, (i+2)_3; \\
& R^i_3=i_4\, (i+1)_2\, (i+2)_4;\\
&R^i_4=i_0\, (i+1)_3\, (i+2)_0. \\
    \end{aligned}
  $ 
\par}
\end{multicols}

\noindent We then form the following six directed cycle of length $5n$:

\smallskip
{\centering
  $ \displaystyle
    \begin{aligned} 
& C^4=U_0N^2_0N^4_0\cdots N^{n-2}_0 U_1N^2_1N^4_1\cdots N^{n-2}_1 U_2N^2_2N^4_2\cdots N^{n-2}_2 U_3N^2_3N^4_3\cdots N^{n-2}_3;\\
&C^5=W_0O^2_0O^4_0\cdots O^{n-2}_0 W_1O^2_1O^4_1\cdots O^{n-2}_1 W_2O^2_2O^4_2\cdots O^{n-2}_2 W_3O^2_3O^4_3\cdots O^{n-2}_3;\\ 
&C^6=X_0P^2_0P^4_0\cdots P^{n-2}_0 X_1P^2_1P^4_1\cdots P^{n-2}_1 X_2P^2_2P^4_2\cdots P^{n-2}_2 X_3P^2_3P^4_3\cdots P^{n-2}_3;\\
&C^7=Y_0Q^2_0Q^4_0\cdots Q^{n-2}_0 Y_1Q^2_1Q^4_1\cdots Q^{n-2}_1 Y_2Q^2_2Q^4_2\cdots Q^{n-2}_2 Y_3Q^2_3Q^4_3\cdots Q^{n-2}_3;\\
&C^8=Z_0R^2_0R^4_0\cdots R^{n-2}_0 Z_1R^2_1R^4_1\cdots R^{n-2}_1 Z_2R^2_2R^4_2\cdots R^{n-2}_2 Z_3R^2_3R^4_3\cdots R^{n-2}_3\\
&~~~~~~~~Z_4R_4^2R_4^4\cdots R_4^{n-2}.\\
    \end{aligned}
  $ 
\par}
\smallskip

It is routine to verify that the directed cycles in  $D=\{C^0, C^1, \ldots, C^8\}$ are pairwise-arc disjoint and thus, the set $D$ is a decomposition of $\vec{C}_n \wr K^*_5$ into directed hamiltonian cycles. \end{proof}

We conclude this section with the proof of Theorem \ref{thm:1}.

\noindent{\bf Proof of Theorem \ref{thm:1}:} Lemma \ref{lem:redCn}, in conjunction with Proposition \ref{prop:complete2} and Lemma \ref{lem:spec35}, implies sufficiency. Conversely, there exists exactly one strict hamiltonian decomposable digraph on two vertices, namely $G=\vec{C}_2=K^*_2$. This means that, if $(n,m) =(2,3)$, then $G\wr K^*_m=K^*_6$. It follows from \cite{Berm2} that $K^*_6$ is not hamiltonian decomposable.  $\square$

\section{Hamiltonian decomposition of $\vec{C}_n \wr\vec{C}_m$ with $m \geqslant 4$}\label{S:Cm}

We conduct our investigation in two stages. First, we consider the case $m$ even followed by the case $m$ odd. 

\begin{proposition}
\label{prop:neven}
Let $n\geqslant 4$ and $m\geqslant 6$ be even integers. The digraph $\vec{C}_n \wr \vec{C}_m$ is hamiltonian decomposable.
\end{proposition}

\begin{proof} 
We will construct a decomposition of $\vec{C}_n \wr \vec{C}_m$ into $m+1$ directed hamiltonian cycles. 

If $i$ is even or $i =n-1$, we embed $\vec{C}_m$ into $V_i$ as follows: $i_{0}\, i_{1}\cdots\, i_{m-1}\, i_{0}$. If $i$ is odd and $i\leqslant n-3$, then we embed $\vec{C}_m$ into $V_i$ as follows: $i_{0}\, i_{m-1}\, i_{m-2}\cdots\, i_{1}\, i_{0}$.  

First, we construct two directed hamiltonian cycles. For each $j \in \mathds{Z}_m$, we construct the following dipath of length $2n$: $P_j=0_j\, 0_{j+1}\, 1_{j+1}\, 1_{j}\, 2_j\, 2_{j+1}\cdots\, (n-2)_j\, (n-2)_{j+1}\, (n-1)_{j+1}\, (n-1)_{j+2}\, 0_{j+2}$.

Next, we let $C^0=P_0P_2\cdots P_{m-4}P_{m-2}$ and $C^1=P_1P_3P_5 \cdots P_{m-3}P_{m-1}$. Each of $C^0$ and  $C^1$ is a directed hamiltonian cycle of $\vec{C}_n \wr \vec{C}_m$. Moreover, if $i \neq j$, then $P_i$ and $P_j$ are arc-disjoint. This means that $C^0$ and $C^1$ are also arc-disjoint. 

The directed cycles $C^0$ and $C^1$ jointly use all arcs of difference 0 and all horizontal arcs exactly once. The digraph obtained by removing the arcs of $C^0$ and $C^1$ from  $\vec{C}_n\wr\vec{C}_m$ is necessarily isomorphic to $\vec{C}_n \times K^*_m$. By Theorem \ref{thm:cat}, $\vec{C}_n \times K^*_m$ is hamiltonian decomposable. Therefore, $\vec{C}_n \wr \vec{C}_m$ is hamiltonian decomposable.\end{proof}

We point out that the construction given in Proposition \ref{prop:neven} does not apply to the case $n=2$. This exception stems from the fact that it is not known whether $\vec{C}_2 \times K^*_m$ is hamiltonian decomposable when $m$ is even. Using a computer, we have affirmed this statement for all even $m$ such that $6\leqslant m \leqslant 16$. See Appendix B of \cite{thesis}.

In addition, observe that the case $m=4$ is not addressed in Proposition \ref{prop:neven} because Theorem \ref{thm:cat} does not apply to this case. Therefore, a separate construction is given for this case below.

\begin{lemma}
\label{lem:m4}
Let $n$ be an even integer. The digraph $\vec{C}_n\wr \vec{C}_4$ is hamiltonian decomposable. 
\end{lemma}

\begin{proof} We will consider two cases. 

\noindent \underline{Case 1}: $n=2$. The digraph $\vec{C}_4$ is embedded into $V_0$ and $V_1$ as follows: $0_0\, 0_1\, 0_2\, 0_3\, 0_0$ and $1_0\, 1_1\, 1_2\, 1_3\, 1_0$, respectively. Using this embedding, we construct five directed hamiltonian cycles:

 \smallskip
{\centering
  $ \displaystyle
    \begin{aligned} 
&C^0=0_0\,1_0\,0_3\, 1_2\, 0_1 \,1_1\, 0_2\, 1_3\, 0_0;&C^2=0_0 \,0_1 \,1_0 \,1_1 \,0_3\, 1_3\, 0_2\, 1_2\, 0_0;  &~~~C^4=0_0 \,1_1 \,0_1 \,1_2\, 1_3\, 1_0\, 0_2\, 0_3\, 0_0.\\
&C^1=0_0\, 1_2\, 0_2\, 1_0\,0_1 \,1_3\, 0_3\, 1_1\, 0_0; &C^3=0_0 \,1_3 \,0_1\, 0_2\, 1_1\, 1_2\, 0_3\, 1_0\, 0_0;&
    \end{aligned}
  $ 
\par}
\smallskip

\noindent It can easily be verified that each arc of $\vec{C}_2 \wr \vec{C}_4$ is used by exactly one directed cycle in $D=\{C^0, C^1, C^2, C^3, C^4\}$. As a result, the set $D$ is a decomposition of $\vec{C}_2 \wr \vec{C}_4$ into directed hamiltonian cycles. 

\noindent \underline{Case 2}: $n \geqslant 4 $. For $i \in \{5,7, \ldots, n-1\}$, $\vec{C}_4$ is embedded into $V_i$ as follows: $i_0\, i_3\, i_2\, i_1\, i_0$. Otherwise, for all other values of $i$, $\vec{C}_4$ is embedded into $V_i$ as follows: $i_0\, i_1\, i_2\, i_3\, i_0$. We now construct the following 16 dipaths:

\begin{figure} [h!]
    \begin{multicols}{4}
\smallskip
{\centering
  $ \displaystyle
    \begin{aligned} 
&U_0=0_0\, 1_2\,2_0\,3_1\,4_1;\\
&U_1=0_1\, 1_3\,2_2\,3_2\,4_2;\\
&U_2=0_2\, 1_1\,2_3\,3_3\,4_3;\\
&U_3=0_3\, 1_0\,2_1\,3_0\,4_0;\\
    \end{aligned}
  $ 
\par}
{\centering
  $ \displaystyle
    \begin{aligned} 
 &W_0=0_0\, 1_3\, 1_0\,2_2\,2_3\,3_1\,4_2;\\
&W_1=0_2\, 0_3\, 1_2\, 2_1\,3_3\, 3_0\,4_1;\\
&W_2=0_1\, 1_1\,2_0\,3_2\, \,4_0;\\
&X_0=0_3\, 0_0\, 1_1\,2_2\,3_3\,4_2;\\
    \end{aligned}
  $ 
\par}
{\centering
  $ \displaystyle
    \begin{aligned} 
&X_1=0_2\, 1_2\, 1_3\, 2_0\,2_1\, 3_2\,4_1;\\
&X_2=0_1\, 1_0\,2_3\,3_0\, 3_1\,4_3;\\
&Y_0=0_0\, 0_1\,1_2\, 2_2\,3_0\,4_2;\\
&Y_1=0_2\, 1_0\, 1_1\, 2_1\,3_1\, 3_2\,4_3;\\
    \end{aligned}
  $ 
\par}
{\centering
  $ \displaystyle
    \begin{aligned} 
&Y_2=0_3\, 1_3\,2_3\,2_0\, 3_3\,4_0;\\
&Z_0=0_0\, 1_0\,2_0\, 3_0\,4_3;\\
&Z_1=0_3\, 1_1\, 1_2\, 2_3\,3_2\, 3_3\,4_1;\\
&Z_2=0_1\, 0_2\,1_3\,2_1\, 2_2\, 3_1\,4_0.\\
    \end{aligned}
  $ 
\par}
\smallskip
\end{multicols}
\end{figure}
\vspace{-1em}

Observe that, if $n=4$, then $4_j=0_j$. Therefore, if $n=4$, then we can form the following five directed cycles of length 16: 

 \smallskip
{\centering
  $ \displaystyle
    \begin{aligned} 
&C^0=U_0U_1U_2U_3; &C^1=W_0W_1W_2; &~~C^2=X_0X_1X_2; &C^3=Y_0Y_1Y_2; &~~C^4=Z_0Z_1Z_2.\\
    \end{aligned}
  $ 
\par}
\smallskip

\noindent It is routine to verify that each directed cycle in  $D=\{C^0, C^1, C^2, C^3, C^4\}$ are pairwise arc-disjoint. This means that $D$ is a decomposition of $\vec{C}_4\wr \vec{C}_4$ into directed hamiltonian cycles. 

Otherwise, if $n\geqslant 6$, we also form the following 16 dipaths:

{\centering
  $
    \begin{aligned} 
 &L_0=4_1\,5_3\,6_1\,7_3\,8_1\cdots (n-2)_1\, (n-1)_3\, 0_1;\\
&L_1=4_2\,5_0\,6_2\,7_0\,8_2\cdots (n-2)_2\, (n-1)_0\, 0_2;\\
&L_2=4_3\,5_1\,6_3\,7_1\,8_3\cdots (n-2)_3\, (n-1)_1\, 0_3;\\
&L_3=4_0\,5_2\,6_0\,7_2\,8_0\cdots (n-2)_0\, (n-1)_2\, 0_0;\\
&M_0=4_2\, 4_3\, 5_3\, 5_2\, 6_2\, 6_3\, 7_3\, 7_2 \cdots (n-2)_2\, (n-2)_3\, (n-1)_3\, (n-1)_2\, 0_0;\\
&M_1=4_1\, 5_0\, 6_1\, 7_0\, 8_1\, 9_0 \cdots (n-2)_1\, (n-1)_0\, 0_1;\\
&M_2=4_0\, 5_1\, 6_0\, 7_1\, 8_0\, 9_1 \cdots (n-2)_0\, (n-1)_1\, 0_0;\\ 
&N_0=4_2\, 5_1\, 6_2\, 7_1\, 8_2\, 9_1 \cdots (n-2)_2\, (n-1)_1\, 0_2;\\
&N_1=4_1\, 5_2\, 6_1\, 7_2\, 8_1\, 9_2\cdots (n-2)_1\, (n-1)_2\, 0_1;\\
&N_2=4_3\, 4_0\, 5_0\, 5_3\, 6_3\, 6_0\,7_0\, 7_3\, \ldots (n-2)_3\, (n-2)_0\, (n-1)_0\, (n-1)_3\, 0_3;\\
&O_0=4_2\, 5_3\, 6_2\, 7_3\, 8_2\, 9_3\cdots (n-2)_2\, (n-1)_3\, 0_2;\\
&O_1=4_3\, 5_2\, 6_3\, 7_2\, 8_3\, 9_2\cdots (n-2)_3\, (n-1)_2\, 0_3;\\
    \end{aligned}
  $ 
\par}
{\centering
  $
    \begin{aligned} 
    &O_2=4_0\, 4_1\, 5_1\, 5_0\, 6_0\, 6_1 \,7_1\, 7_0\, \ldots (n-2)_0\, (n-2)_1\, (n-1)_1\, (n-1)_0\, 0_0;\\ 
    &P_0=4_3\, 5_0\, 6_3\, 7_0\, 8_3\, 9_0\cdots (n-2)_3\, (n-1)_0\, 0_3;\\
&P_1=4_1\, 4_2\, 5_2\, 5_1\, 6_1\, 6_2 \,7_2\, 7_1\, \ldots (n-2)_1\, (n-2)_2\, (n-1)_2\, (n-1)_1\, 0_0;\\
&P_2=4_0\, 5_3\, 6_0\, 7_3\, 8_0\, 9_3\cdots (n-2)_0\, (n-1)_3\, 0_0.
    \end{aligned}
  $ 
\par}

All $L_i$ are of length $n-4$. Furthermore, dipaths $M_0, N_2, O_2$, and $P_1$ are of length $2(n-4)$. The remaining dipaths are of length $n-4$. We can then form the following directed walks:

\smallskip
{\centering
  $ \displaystyle
    \begin{aligned} 
&C^0=U_0L_0U_1L_1U_2L_2U_3L_3; &&C^2=X_0N_0X_1N_1X_2N_2;  &&C^4=Z_0P_0Z_1P_1Z_2P_2.\\
&C^1=W_0M_0W_1M_1W_2M_2; &&C^3=Y_0O_0Y_1O_1Y_2O_2; &&
    \end{aligned}
  $ 
\par}
\smallskip

\noindent Let $D=\{C^0, C^1, C^2, C^3, C^4\}$. It is tedious but straightforward to verify that each directed walk of $D$ is a directed hamiltonian cycle and that these directed cycles are pairwise arc-disjoint. Therefore, the set $D$ is a hamiltonian decomposition of $\vec{C}_n \wr \vec{C}_4$. \end{proof}

We now proceed with the case of Conjecture \ref{conj:main} where $G=\vec{C}_n$ and $H=\vec{C}_m$, $n$ is even, $m\geqslant 5$, and $m$ is odd. The case $m=3$ is omitted because we show that $\vec{C}_n \wr \vec{C}_3$ is not hamiltonian decomposable for all even $n$ in Section \ref{sec:contra}. 

\begin{proposition}
\label{prop:odd}
Let $n$ be an even integer and $m\geqslant 5$ be an odd integer. The digraph $\vec{C}_n \wr \vec{C}_m$ is hamiltonian decomposable.  
\end{proposition}

\begin{proof} To construct a directed hamiltonian decomposition of $\vec{C}_n \wr \vec{C}_m$, we consider two cases.  In both cases, we use the following embedding of $\vec{C}_m$ into each $V_i$. If $i$ is even or $i =n-1$, we embed $\vec{C}_m$ into $V_i$ as follows: $i_{0}\, i_{1}\, i_2\cdots\, i_{m-1}\, i_0$. If $i$ is odd and $i\leqslant n-3$, then we embed $\vec{C}_m$ into $V_i$ as follows: $i_{0}\, i_{m-1}\cdots\, i_{2}\, i_{1}\, i_0$.  

\noindent \underline{Case 1}: $m \equiv 1\ (\textrm{mod} \ 4)$. First, we construct two specific directed hamiltonian cycles of  $\vec{C}_n \wr \vec{C}_m$. 
 
If $n=2$, we let

\smallskip
{\centering
  $ \displaystyle
    \begin{aligned} 
&C^0=0_0\, 0_1\, 1_1\, 1_2\, 0_2\, 0_3\, 1_3\, 1_4\cdots\, 0_{m-3}\, 0_{m-2}\, 1_{m-2}\, 0_{m-1}\, 1_{m-1}\, 1_0\, 0_0;\\
&C^1=0_1\, 0_2\, 1_2\, 1_3\, 0_3\, 0_4\, 1_4\, 1_5\cdots\, 0_{m-4}\, 0_{m-3}\, 1_{m-3}\, 1_{m-2}\, 1_{m-1}\, 0_{m-2}\, 0_{m-1}\, 0_0\, 1_0\,  1_1\, 0_1. 
     \end{aligned}
  $ 
\par}
 \smallskip
 
\noindent Otherwise, if $n\geqslant 4$, we construct a set of $m$ dipaths. For each $j$ such that $0 \leqslant j \leqslant m-5$, we construct the following dipath:
 $P_j=0_j\, 0_{j+1}\, 1_{j+1}\, 1_{j}\, 2_j\, 2_{j+1}\cdots\, (n-2)_j\, (n-2)_{j+1}\, (n-1)_{j+1}\, (n-1)_{j+2}\, 0_{j+2}$.  Next, we build the following four dipaths:
 
 \smallskip
 {\centering
  $ \displaystyle
    \begin{aligned} 
  &P_{m-4}&= \quad &0_{m-4}\, 0_{m-3}\, 1_{m-3}\, 1_{m-4}\, 2_{m-4}\, 2_{m-3} \cdots (n-2)_{m-4}\, (n-2)_{m-3}\, (n-1)_{m-3}\, (n-1)_{m-2}\,\\
  &&& (n-1)_{m-1}\,  0_{m-2};\\   
 &P_{m-3}&= \quad &0_{m-3}\, 0_{m-2}\, 1_{m-2}\, 1_{m-3}\, 2_{m-3}\, 2_{m-2}\cdots (n-2)_{m-3}\, (n-2)_{m-2}\, (n-1)_{m-2}\, 0_{m-1};\\
 &P_{m-2}&= \quad &0_{m-2}\, 0_{m-1}\, 0_0\,  1_{0}\, 1_{m-1}\, 1_{m-2}\, 2_{m-2}\,  2_{m-1}\, 2_0\cdots (n-2)_{m-2}\, (n-2)_{m-1}\,(n-2)_{0}\, (n-1)_{0}\, \\ 
&& & (n-1)_1\, 0_1;\\
 &P_{m-1}&= \quad &0_{m-1}\, 1_{m-1}\, 2_{m-1} \cdots (n-2)_{m-1}\, (n-1)_{m-1}\, (n-1)_{0}\,  0_{0}. 
    \end{aligned}
      $ 
\par}
\smallskip

\noindent  We then concatenate the dipaths constructed above as follows: $C^0=P_0P_2P_4 \cdots P_{m-3}P_{m-1}$ and $C^1=P_1P_3P_5 \cdots P_{m-4}P_{m-2}$. In both cases ($n=2$ and $n \geqslant 4$),  $C^0$ and $C^1$ jointly use all horizontal arcs and all arcs of difference 0 except for arcs $((n-1)_{m-2}, 0_{m-2})$ and $((n-1)_{m-1}, 0_{m-1})$. Moreover, the directed walks $C^0$ and $C^1$ have no repeated vertices other than their respective endpoints. Therefore, both $C^0$ and $C^1$ are directed hamiltonian cycles of $\vec{C}_n \wr \vec{C}_m$. Lastly, we note that $C^0$ and $C^1$ are arc-disjoint. 
 
Next, we construct $m-1$ 2-factors using $n$-tuples of permutations, as described in Notation  \ref{not:Sm}. Let  $\mu=(0\,\,1)(2\,\,3)(4\,\,5)\cdots(m-3\,\, m-2)$, $\sigma=(1\,\,2)(3\,\,4)(5\,\,6) \cdots  (m-4\,\, m-3)(m-1\,\,0)$, \linebreak $\pi_1=(0 \,\,1\,\, 2 \cdots \,\,m-1)$, and $\pi_k=\pi_1^k$. If $n=2$, we construct the following directed 2-factors: $C^2=(\pi_2, \mu)$ and $C^3=(\pi_3, \sigma)$. Observe that

    \smallskip
{\centering
  $ \displaystyle
    \begin{aligned} 
   &\pi_2\mu=(0\,\,3\,\,4\,\, 7\,\, 8\, 11\,\, 12\,\, 15\,\, 16\, \cdots\, m-5\, m-2\,\, 1\,\, 2\,\, 5\,\, 6\,\, 9\,\, 10\, \cdots\,\, m-3\,\, m-1);\\
 &\pi_3\sigma=(0\,\,4\,\,8\,\,12\,\, \cdots\,\, m-5\,\, m-2\,\, 2\,\,6\,\, \cdots\,\, m-3\,\, m-1\,\,1\,\, 3\,\, 5\,\,7\,\, \cdots\,\, m-4). 
        \end{aligned}
  $ 
\par}
\smallskip
  
\noindent Otherwise, if $n>2$, we construct the following directed 2-factors:
 
   \smallskip
{\centering
  $ \displaystyle
    \begin{aligned} 
 &C^2=(\pi_1, \pi_{-1}, \pi_1, \pi_{-1}, \ldots, \pi_1, \pi_{-1}, \pi_2, \mu); &C^3=(\pi_{-1}, \pi_{1}, \pi_{-1}, \pi_{1}, \ldots, \pi_{-1}, \pi_{1}, \pi_3, \sigma).\\
      \end{aligned}
  $ 
\par}
\smallskip

\noindent Note that $\pi_1 \pi_{-1} \pi_1 \pi_{-1} \cdots \pi_1\pi_{-1}\pi_2\mu=\pi_2\mu$ and $\pi_{-1}\pi_{1} \pi_{-1} \pi_{1} \cdots \pi_{-1} \pi_{1} \pi_3 \sigma=\pi_3\sigma$.

Since $T(\pi_2\mu)=T(\pi_3\sigma)=1$, $C^2$ and $C^3$ are directed hamiltonian cycles for all even $n \geqslant 2$. Moreover, we see that $\{((n-1)_j, 0_{j^\mu}), ((n-1)_j, 0_{j^\sigma}) \ |\ j \in \mathds{Z}_m\}\subseteq A(C^2)\cup A(C^3)$. This means that each arc of difference 0 in $\vec{C}_n\wr \vec{C}_m$ appears precisely once in $\{C^0, C^1, C^2, C^3\}$.  

Next, we construct  $m-4$ directed 2-factors. For each $i \in \mathds{Z}_m$ such that $i\not\in \{0,\pm1, -2\}$, we create the following directed 2-factor: $C^{i+2}=(\pi_{i}, \pi_{-i}, \ldots, \pi_{i}, \pi_{-i}, \pi_{-i+1}, \pi_{i})$. We also let  $C^{m}=(\pi_{-2}, \pi_{2}, \ldots, \pi_{-2}, \pi_2, \pi_{1}, \pi_{-2})$. Note that  $T(\pi_{i}\pi_{-i}\cdots \pi_{i} \pi_{-i}\pi_{-i+1}\pi_{i})=T(\pi_1)=1$ and $T(\pi_{-2}\pi_{2} \cdots \pi_{-2}$ $\pi_2 \pi_{1} \pi_{-2})=T(\pi_{-1})=1$. Therefore, each directed 2-factor $C^i$, where $0 \leqslant i \leqslant m$, is a directed hamiltonian cycle. Lastly, it can be verified that each arc of each difference $d \in \mathds{Z}_m$ and each horizontal arc appears exactly once in  $D=\{C^0, C^1, \ldots, C^m\}$. Therefore, the set $D$ is a directed hamiltonian decomposition of $\vec{C}_n \wr \vec{C}_m$. 
  
\noindent \underline{Case 2}: $m \equiv 3\ (\textrm{mod} \ 4)$. Hence $m \geqslant 7$. First, we construct the following six specific dipaths:

 \smallskip
{\centering
  $ \displaystyle
    \begin{aligned} 
&P_0=0_0\, 0_1\, 0_2\, 1_2\, 1_1\, 1_0\, 2_0\, 2_1\, 2_2\, 3_2\cdots(n-2)_0; &&P'_1= (n-2)_1\, (n-1)_1\, (n-1)_2\, 0_2;  \\
&P_0'= (n-2)_0\, (n-2)_1\, (n-2)_2\, (n-1)_2\,(n-1)_3\, (n-1)_4\, 0_3; &&P_2=0_2\, 0_3\, 1_3\, 1_2\, 2_2\cdots\, (n-2)_2; \\
&P_1=0_1\, 1_1\, 2_1\cdots\, (n-2)_1; &&P'_2=(n-2)_2\, (n-2)_3\, (n-1)_3\, 0_4.\\ 
    \end{aligned}
  $ 
\par}
\smallskip

\noindent For $j$ odd and $3 \leqslant j \leqslant m-2$, we construct:

 \smallskip
{\centering
  $ \displaystyle
    \begin{aligned} 
& P_j&&=&&0_j\, 0_{j+1}\, 0_{j+2}\,  1_{j+2}\, 1_{j+1}\, 1_j\,  2_{j}\, 2_{j+1}\, 2_{j+2}\, 3_{j+2}\, 3_{j+1}\, 3_{j}\,  \cdots\, (n-2)_j;\\
& P'_j&&=&& (n-2)_j\, (n-2)_{j+1}\,(n-2)_{j+2}\, (n-1)_{j+2}\, 0_{j+3}.\\
    \end{aligned}
  $ 
\par}
\smallskip

\noindent For $j$ even and $4 \leqslant j \leqslant m-1$, we construct:

 \smallskip
{\centering
  $ \displaystyle
    \begin{aligned} 
& P_j=0_j\, 1_{j} \, 2_{j} \cdots\, (n-2)_j; &P'_j=  (n-2)_j\,(n-1)_{j}\, (n-1)_{j+1}\, (n-1)_{j+2}\, 0_{j+1}.
    \end{aligned}
  $ 
\par}
\smallskip

\noindent  If $n=2$, then $P_j$ is a dipath of length 0 for all $j \in \mathds{Z}_m$ and we thus form the following directed walks:

 \smallskip
{\centering
  $ \displaystyle
    \begin{aligned}
C^0=P'_0P'_3P'_6P'_7P'_{10}P'_{11}\cdots P'_{m-4}P'_{m-1}; C^1=P'_1P'_2P'_4P'_5P'_{8}P'_{9} \cdots P'_{m-3}P'_{m-2}.
    \end{aligned}
  $ 
\par}
\smallskip

\noindent Next, consider the case $n \geqslant 4$. We form the following directed walks:
 
 \smallskip
{\centering
  $ \displaystyle
    \begin{aligned}
&C^0=P_0P_0'P_3P_3'P_6P_6'P_7P_7'P_{10}P_{10}'P_{11}P'_{11}\cdots P_{m-4}P_{m-4}'P_{m-1}P_{m-1}';\\
&C^1=P_1P_1'P_2P_2'P_4P_4'P_5P_5'P_{8}P_8'P_{9}P'_9 \cdots P_{m-3}P_{m-3}'P_{m-2}P_{m-2}'.
    \end{aligned}
  $ 
\par}
\smallskip

\noindent It it tedious but straightforward to verify that $C^0$ and $C^1$ are arc-disjoint directed hamiltonian cycles. 
 
Let $\mu=(1\,\,2)(4\,\,5)$ and $\sigma=(2\,\,3)(6\,\,7)(8\,\,9) \cdots (m-3\,\, m-2)(m-1\,\,0)$.  If $n=2$, we construct the following directed 2-factors: $C^2=(\pi_2, \mu)$,  $C^3=(\pi_3, \sigma)$. Observe that

    \smallskip
{\centering
  $ \displaystyle
    \begin{aligned} 
   &\pi_2\mu=(0\,\,1\,\,3\,\,4\,\,6\,\,8\,\,10\,\, \cdots\,\, m-1\,\, 2\,\, 5\,\, 7\,\, \cdots\,\, m-2);\\
   &\pi_3\sigma=(0\,\,2\,\, 5\,\, 9\,\, 13\,\, 17\,\, \cdots\,\, m-2\,\, 1\,\, 4\,\, 6\,\, 8\,\, 10\,\, \cdots\,\, m-1\,\, 3\,\, 7\,\, \cdots\,\, m-4).  
        \end{aligned}
  $ 
\par}
\smallskip
 
\noindent If $n\geqslant 4$, we construct the following directed 2-factors in the same fashion as Case 1: 
  
      \smallskip
{\centering
  $ \displaystyle
    \begin{aligned}  
  &C^2=(\pi_1, \pi_{-1}, \pi_1, \pi_{-1}, \ldots, \pi_1, \pi_{-1}, \pi_2, \mu); &C^3=(\pi_{-1}, \pi_{1}, \pi_{-1}, \pi_{1}, \ldots, \pi_{-1}, \pi_{1}, \pi_3, \sigma). 
         \end{aligned}
  $ 
\par}
\smallskip

\noindent Observe that $\pi_1\pi_{-1}\pi_1\pi_{-1} \cdots\pi_1\pi_{-1}\pi_2 \mu=\pi_2\mu$ and $\pi_{-1} \pi_{1} \pi_{-1} \pi_{1} \cdots \pi_{-1} \pi_{1} \pi_3\sigma=\pi_3\sigma$. We then note that $T(\pi_2\mu)=T(\pi_3\sigma)=1$. Therefore, both $C^2$ and $C^3$ are directed hamiltonian cycles for all even $n \geqslant 2$. 

Next, we construct $m-4$ directed 2-factors. For each $i \in \mathds{Z}_m$ and $i\not\in \{0,\pm1, -2\}$, we let: $C^{i+2}=(\pi_{i}, \pi_{-i}, \ldots, \pi_i, \pi_{-i},  \pi_{-i+1}, \pi_{i})$. Furthermore, we let $C^{m}=(\pi_{-2}, \pi_{2}, \ldots, \pi_{-2}, \pi_2, \pi_{1}, \pi_{-2})$. Note  that  $T(\pi_{i}\pi_{-i}\cdots\pi_i \pi_{-i}\pi_{-i+1}\pi_{i})$$=T(\pi_{1})=T(\pi_{-2}\pi_{2}\cdots\pi_{-2}\pi_{2}\pi_{1}\pi_{-2})=T(\pi_{-1})=1$. Therefore, each directed 2-factor $C^i$, where $0 \leqslant i \leqslant m$, is a directed hamiltonian cycle. Lastly, it can be verified that each arc of each difference $d \in \mathds{Z}_m$, and each horizontal arc, appears precisely once in $D=\{C^0, C^1, \ldots, C^m\}$. Therefore, the set $D$ is a directed hamiltonian decomposition of $\vec{C}_n \wr \vec{C}_m$.  \end{proof}

We conclude this section with the proof of Theorem \ref{thm:2}.

\noindent{\bf Proof of Theorem \ref{thm:2}:}
\begin{enumerate} [label=\textbf{(S\arabic*)}]
\item The result follows from Lemma \ref{lem:redCn} and Proposition \ref{prop:neven}. 
\item The result follows from Lemmas \ref{lem:redCn} and \ref{lem:m4}.
\item The result follows from Lemma \ref{lem:redCn} and Proposition \ref{prop:odd}. 
\end{enumerate}

Lastly, Propositions \ref{prop:notm2} and \ref{prop:notm3} imply that $\vec{C}_n\wr \vec{C}_{2}$ and $\vec{C}_n\wr \vec{C}_{3}$, respectively,  are not hamiltonian decomposable for all even $n$. See Section \ref{sec:contra} for Proposition \ref{prop:notm3} and its proof. $\square$

\section{The digraph $\vec{C}_n \wr \vec{C}_3$ is not hamiltonian decomposable}
\label{sec:contra}
In this section, we will show that $\vec{C}_n\wr \vec{C}_3$ is not hamiltonian decomposable for all even $n\geqslant 2$. This is another non-trivial exception to Conjecture \ref{conj:main}. The proof of this statement requires two cases. In Subsection \ref{sec:preA}, we will first introduce key notation and show that we can partition the set of all 2-factorizations of $\vec{C}_n\wr \vec{C}_3$ into two particular sets. Then, in Subsections \ref{sec:fix} and \ref{sec:type2}, we show that neither set contains a directed 2-factorization of  $\vec{C}_n\wr \vec{C}_3$ that is also a hamiltonian decomposition. 

\subsection{Preliminaries}
\label{sec:preA}

We begin by introducing the following assumption made throughout this section. 

\begin{assumption} \rm
\label{not:C3i}
In $\vec{C}_n \wr \vec{C}_3$, we embed $\vec{C}_3$ into $V_i$ as $i_0\, i_1\, i_2\, i_0$ for all $i \in \mathds{Z}_n$ and we let $C^i_3=i_0\, i_1\, i_2\,i_0$. 
\end{assumption}

We now demonstrate that all 2-factors of $\vec{C}_n \wr \vec{C}_3$ must contain the same number of horizontal arcs from each $C_3^j$.

\begin{lemma}
\label{lem:hor}
Let  $F$ be a directed 2-factor of $\vec{C}_n \wr \vec{C}_3$. Then there exists an integer $k$ such that $0 \leqslant k \leqslant 3$ and $|A(F)\cap A(C^i_3)|=k$ for all $i \in \mathds{Z}_n$. 
\end{lemma}

\begin{proof} Suppose that $F$ is a directed 2-factor of $\vec{C}_n \wr \vec{C}_3$ such that $|A(F)\cap A(C^i_3)|$ is not constant. Then, without loss of generality, we may assume that, for some $i \in \mathds{Z}_n$, $F$ contains $k_1$ horizontal arcs from $C^i_3$ and $k_2$ horizontal arcs from $C^{j+1}_3$ and that $0 \leqslant k_1<k_2 \leqslant 3$. 

If $F$ uses $k_1$ horizontal arcs from $C^i_3$, then $F$ contains $3-k_1$ non-horizontal arcs with tail in $V_i$. By our assumption, $F$ also uses $k_2$ horizontal arcs from $C^{i+1}_3$. This means that $F$ must contain $3-k_2$ non-horizontal arcs with head in $V_{i+1}$ and these must be the non-horizontal arcs with tail in $V_i$. Since $3-k_1\neq 3-k_2$, we have a contradiction. \end{proof} 

As a result of Lemma \ref{lem:hor}, we can introduce the following definition. 

\begin{definition}
\label{defn:types}
\rm
Let $0\leqslant k\leqslant 3$. A \textit{type-$k$}  2-factor of $\vec{C}_n \wr \vec{C}_3$ is a directed 2-factor that contains $k$ horizontal arcs from each $C^i_3$.  A \textit{type-I} 2-factorization of $\vec{C}_n \wr \vec{C}_3$ is  comprised of three type-1 2-factors and one type-0 2-factor. A \textit{type-II} 2-factorization of $\vec{C}_n\wr \vec{C}_3$ is comprised of one type-2 2-factor, one type-1 2-factor, and two type-0 2-factors. 
\end{definition}

In Definition \ref{defn:types}, we described all possible 2-factorizations of interest. If a directed 2-factorization of $\vec{C}_n\wr \vec{C}_3$ contains a type-3 2-factor, then this 2-factorization is not hamiltonian since a type-3 2-factor is necessarily the digraph $n\vec{C}_3$. 

\begin{proposition} 
\label{prop:notm3}  
Let $n$ be a positive even integer. The digraph $\vec{C}_n \wr \vec{C}_3$ is not hamiltonian decomposable.
\end{proposition}

\begin{proof} As per the observation above, a directed 2-factorization of $\vec{C}_n\wr \vec{C}_3$ that is also a hamiltonian decomposition must be a type-I or a type-II 2-factorization. Propositions \ref{thm:typ1} and \ref{thm:typ2}, respectively, show that  $\vec{C}_n \wr \vec{C}_3$ admits neither a type-I nor a type-II  2-factorization that is also a hamiltonian decomposition. \end{proof}

Before we proceed, we give some further notation and definitions used to describe the constructions given in Subsections \ref{sec:fix} and \ref{sec:type2}. 

\begin{notation} \rm
\label{not:induced}
By $L_i$, we denote the subdigraph of $\vec{C}_n \wr \vec{C}_3$ induced by the set of vertices $V_i \cup V_{i+1}$ where $i \in \{0,1,2, \ldots, n-2\}$; the subdigraph $L_{n-1}$ is the subdigraph of $\vec{C}_n \wr \vec{C}_3$  induced by  $V_{n-1} \cup V_{0}$. 
\end{notation}

\begin{definition} \rm
\label{defn:inter}
Let $\mathcal{F}=\{F_0, F_1, F_2, F_3\}$ be a type-I or a type-II 2-factorization of $\vec{C}_n \wr \vec{C}_3$. By $F_j[i]$, where $j\in \{0,1,2,3\}$, we denote the subdigraph of $\vec{C}_n \wr \vec{C}_3$ with vertex set $V(L_i)$,  where $L_i$ is given in Definition \ref{not:induced}, and arc set $A(F_j)\cap A(L_i)$. 
\end{definition}

If $F_j$ is a type-1 directed 2-factor of  $\vec{C}_n \wr \vec{C}_3$, it follows that $F_j[i]$ contains two horizontal arcs and two non-horizontal arcs of $\vec{C}_n \wr \vec{C}_3$. Furthermore, $F_j[i]$ is the disjoint union of two dipaths that both contain precisely one non-horizontal arc. 

\subsection{Type-I 2-factorization of $\vec{C}_n \wr \vec{C}_3$}
\label{sec:fix}

Our objective is to demonstrate that $\vec{C}_n \wr \vec{C}_3$ does not admit a type-I  $2$-factorization that is also a hamiltonian decomposition. We do so by strategically describing each type-I 2-factorization of $\vec{C}_n \wr \vec{C}_3$ as an $n$-tuple of certain elements of a particular group. Then, we show that this 2-factorization is a hamiltonian decomposition only if the product of these $n$ group elements is one of two specific elements. Lastly, we show that, for all even $n$, no $n$-tuple satisfies this necessary condition. 

We begin by introducing the following definition.

\begin{definition}\rm
\label{defn:rooted}
Let $i \in \mathds{Z}_n$ and $j\in \mathds{Z}_3$. A type-1 directed 2-factor of $\vec{C}_n \wr \vec{C}_3$ that contains the arc $(i_{j}, i_{j+1})$ is called a \textit{$(i, j)$-rooted 2-factor}. 
\end{definition}

Throughout this section, we will be making the following assumption. 

\begin{assumption} \rm
\label{not:Fi}
Let $\mathcal{F}$ be a type-I 2-factorization of $\vec{C}_n \wr \vec{C}_3$ comprised of four $2$-factors $F_0, F_1, F_2,$ and $F_3$. For each $j \in \{0,1,2\}$, we may assume that $F_j$ is the type-1 $(0,j)$-rooted 2-factor of $\mathcal{F}$, and that $F_3$ is the type-0 2-factor of $\mathcal{F}$. We henceforth denote the type-1 2-factorization $\mathcal{F}$ as a 4-tuple $(F_0, F_1, F_2, F_3)$. 
\end{assumption}

Now, we list all possible subdigraphs $F_j[i]$, defined in Definition \ref{defn:inter}, where $j \in \{0,1,2\}$. Define $D_i=\{F_j[i]\ | \ j\in \mathds{Z}_3 \ \textrm{and} \ F_j \ \textrm{is a type-1 2-factor}\}$. Elements of $D_i$ are subdigraphs comprised of pairs of dipaths whose lengths sum to four. There are 18 distinct subdigraphs in $D_i$. We first list all elements of $D_i$ that contain the arc $(i_0, i_1)$. These are listed as the union of two dipaths as follows:

\begin{center}
\begin{tabular}{r l c l} 
$M_0^0[i]$&$=i_0  \,i_1  \,(i+1)_0\,$ &$\cup$& $i_2  \,(i+1)_1  \,(i+1)_2$;\\
$M_1^0[i]$&$=i_0  \,i_1  \, ( j+1)_0 \, (i+1)_1\,$ &$\cup$&$ i_2  \,(i+1)_2$;\\ 
$M_2^0[i]$&$=i_0  \,i_1  \,(i+1)_1\,$&$\cup$&$i_2  \,(i+1)_2  \,(i+1)_0$;\\ 
$M_3^0[i]$&$=i_0  \,i_1  \, ( j+1)_1\,(i+1)_2\,$ &$\cup$&$ i_2  \,(i+1)_0$;\\ 
$M_4^0[i]$&$=i_0  \,i_1  \,(i+1)_2\,$ &$\cup$&$i_2  \,(i+1)_0  \,(i+1)_1$;\\ 
$M_5^0[i]$&$=i_0  \,i_1  \, ( j+1)_2 \,(i+1)_0\,$ &$\cup$&$i_2  \,(i+1)_1$. \\ 
\end{tabular}
\end{center}

\noindent In Figure \ref{fig:Fij}, we illustrate the six elements of $D_i$ given above. These are the only six subdigraphs $F_j[i]$ that contain the arc $(i_0, i_1)$ taken over all possible type-I 2-factors of $\vec{C}_3 \wr \vec{C}_n$. 

\begin{figure} [htpb!]
\begin{center}
\begin{subfigure}[c]{0.32 \textwidth}
\begin{tikzpicture}[
  very thick,
  every node/.style={circle,draw=black,fill=black!90}, inner sep=2]
    every node/.style={circle,draw=black,fill=black!90}, inner sep=1.5, scale=0.75]
  \node (x0) at (0.0,3.0) [label=above:$0$] [label=left:$i$] {};
  \node (x1) at (1.0,3.0) [label=above:$1$]{};
  \node (x2) at (2.0,3.0)[label=above:$2$] {};
  \node (x3) at (3.0,3.0)[draw=gray, fill=gray, label=above:$0$] {};
  \node (y0) at (0.0,2.0) [label=left:$i+1$] {};
  \node (y1) at (1.0,2.0) {};
  \node (y2) at (2.0,2.0) {};
  \node (y3) at (3.0,2.0)[draw=gray, fill=gray] {};
  
  \path [very thick, draw=color8, postaction={very thick, on each segment={mid arrow}}]
	(x0) to (x1)
	(x1) to  (y0)
	(x2) to  (y1)
	(y1) to  (y2);
	
\end{tikzpicture}
\caption{The digraph $M_0^0[i]$.}
\label{fig:sw}
\end{subfigure}
\begin{subfigure}[c]{0.32 \textwidth}
\begin{tikzpicture}[
  very thick,
  every node/.style={circle,draw=black,fill=black!90}, inner sep=2]
    every node/.style={circle,draw=black,fill=black!90}, inner sep=1.5, scale=0.75]
  \node (x0) at (0.0,3.0) [label=above:$0$] [label=left:$i$] {};
  \node (x1) at (1.0,3.0) [label=above:$1$]{};
  \node (x2) at (2.0,3.0)[label=above:$2$] {};
  \node (x3) at (3.0,3.0)[draw=gray, fill=gray, label=above:$0$] {};
  \node (y0) at (0.0,2.0) [label=left:$i+1$] {};
  \node (y1) at (1.0,2.0) {};
  \node (y2) at (2.0,2.0) {};
  \node (y3) at (3.0,2.0)[draw=gray, fill=gray] {};
  
  \path [very thick, draw=color8, postaction={very thick, on each segment={mid arrow}}]
	(x0) to (x1)
	(x1) to  (y0)
	(y0) to  (y1)
	(x2) to  (y2);
\end{tikzpicture}
\caption{The digraph $M_1^0[i]$.}
\label{fig:nsw}
\end{subfigure}
\begin{subfigure}[c]{0.32 \textwidth}
\begin{tikzpicture}[
  very thick,
  every node/.style={circle,draw=black,fill=black!90}, inner sep=2]
    every node/.style={circle,draw=black,fill=black!90}, inner sep=1.5, scale=0.75]
  \node (x0) at (0.0,3.0) [label=above:$0$] [label=left:$i$] {};
  \node (x1) at (1.0,3.0) [label=above:$1$]{};
  \node (x2) at (2.0,3.0)[label=above:$2$] {};
  \node (x3) at (3.0,3.0)[draw=gray, fill=gray, label=above:$0$] {};
  \node (y0) at (0.0,2.0) [label=left:$i+1$] {};
  \node (y1) at (1.0,2.0) {};
  \node (y2) at (2.0,2.0) {};
  \node (y3) at (3.0,2.0)[draw=gray, fill=gray] {};
  
  \path [very thick, draw=color8, postaction={very thick, on each segment={mid arrow}}]
	(x0) to (x1)
	(x1) to  (y1)
	(x2) to  (y2)
	(y2) to  (y3);
\end{tikzpicture}
\caption{The digraph $M_2^0[i]$.}
\label{fig:sw1}
\end{subfigure}

\vspace{5mm}

\begin{subfigure}[c]{0.32 \textwidth}
\begin{tikzpicture}[
  very thick,
  every node/.style={circle,draw=black,fill=black!90}, inner sep=2]
    every node/.style={circle,draw=black,fill=black!90}, inner sep=1.5, scale=0.75]
  \node (x0) at (0.0,3.0) [label=above:$0$] [label=left:$i$] {};
  \node (x1) at (1.0,3.0) [label=above:$1$]{};
  \node (x2) at (2.0,3.0)[label=above:$2$] {};
  \node (x3) at (3.0,3.0)[draw=gray, fill=gray, label=above:$0$] {};
  \node (y0) at (0.0,2.0) [label=left:$i+1$] {};
  \node (y1) at (1.0,2.0) {};
  \node (y2) at (2.0,2.0) {};
  \node (y3) at (3.0,2.0)[draw=gray, fill=gray] {};
  
  \path [very thick, draw=color8, postaction={very thick, on each segment={mid arrow}}]
	(x0) to (x1)
	(x1) to  (y1)
	(y1) to  (y2)
	(x2) to  (y3);
\end{tikzpicture}
\caption{The digraph $M_3^0[i]$.}
\label{fig:nsw1}
\end{subfigure}
\begin{subfigure}[c]{0.32\textwidth}
\begin{tikzpicture}[
  very thick,
  every node/.style={circle,draw=black,fill=black!90}, inner sep=2]
    every node/.style={circle,draw=black,fill=black!90}, inner sep=1.5, scale=0.75]
  \node (x0) at (0.0,3.0) [label=above:$0$] [label=left:$i$] {};
  \node (x1) at (1.0,3.0) [label=above:$1$]{};
  \node (x2) at (2.0,3.0)[label=above:$2$] {};
  \node (x3) at (3.0,3.0)[draw=gray, fill=gray, label=above:$0$] {};
  \node (y0) at (0.0,2.0) [label=left:$i+1$] {};
  \node (y1) at (1.0,2.0) {};
  \node (y2) at (2.0,2.0) {};
  \node (y3) at (3.0,2.0)[draw=gray, fill=gray] {};
  \path [very thick, draw=color8, postaction={very thick, on each segment={mid arrow}}]
	(x0) to (x1)
	(x1) to  (y2)
	(x2) to  (y3)
	(y0) to  (y1);
\end{tikzpicture}
\caption{The digraph $M_4^0[i]$.}
\label{fig:sw2}
\end{subfigure}
\begin{subfigure}[c]{0.32 \textwidth}
\begin{tikzpicture}[
  very thick,
  every node/.style={circle,draw=black,fill=black!90}, inner sep=2]
    every node/.style={circle,draw=black,fill=black!90}, inner sep=1.5, scale=0.75]
  \node (x0) at (0.0,3.0) [label=above:$0$] [label=left:$i$] {};
  \node (x1) at (1.0,3.0) [label=above:$1$]{};
  \node (x2) at (2.0,3.0)[label=above:$2$] {};
  \node (x3) at (3.0,3.0)[draw=gray, fill=gray, label=above:$0$] {};
  \node (y0) at (0.0,2.0) [label=left:$i+1$] {};
  \node (y1) at (1.0,2.0) {};
  \node (y2) at (2.0,2.0) {};
  \node (y3) at (3.0,2.0)[draw=gray, fill=gray] {};
  \path [very thick, draw=color8, postaction={very thick, on each segment={mid arrow}}]
	(x0) to (x1)
	(x1) to  (y2)
	(y2) to  (y3)
	(x2) to  (y1);
\end{tikzpicture}
\caption{The digraph $M_5^0[i]$.}
\label{fig:nsw2}
\end{subfigure}
\end{center}
\caption{All possible subdigraphs $F_j[i]$ of a type-I 2-factor of $\vec{C}_n \wr \vec{C}_3$ containing the arc $(i_0, i_1)$.}
\label{fig:Fij}
\end{figure}
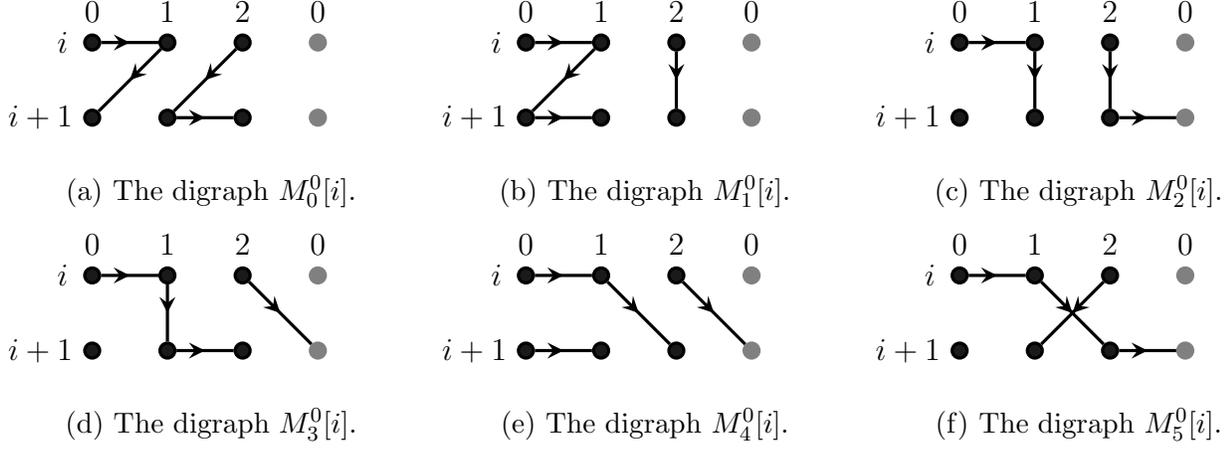

Using these six elements of $D_i$, we construct the remaining 12 elements of $D_i$ by using the following permutation of the elements of $V(\vec{C}_n \wr \vec{C}_3)$. 

\begin{definition}
\label{defn:rhorho}\rm
Define a permutation $\rho: V(\vec{C}_n \wr \vec{C}_3) \mapsto  V(\vec{C}_n \wr \vec{C}_3)$ as follows: $\rho(i_k)=i_{k+1}$ with addition of the indices done modulo 3. Given a dipath of $\vec{C}_n \wr \vec{C}_3$,  $P=v_1\, v_2\, \ldots\, v_\ell$, we let $\rho(P)=\rho(v_1)\, \rho(v_2)\, \ldots\, \rho(v_\ell)$. In addition, if $S$ is a subdigraph of $\vec{C}_n \wr \vec{C}_3$ such that $S=P^0\cup P^1\cup \cdots \cup P^r$, where $P^s$ is a dipath in $\vec{C}_n \wr \vec{C}_3$, then we let $\rho(S)=\rho(P^0)\cup \rho(P^1)\cup \cdots \cup \rho(P^r)$. 
\end{definition}

Using Definition \ref{defn:rhorho}, we construct all elements of $D_i$ as follows:

\begin{center}
$D_i=\{ M_\ell^k[i]=\rho^k(M_\ell^0[i]) \ | \  k=0,1,2\ \textrm{and} \ \ell=0,1, \ldots, 5\}$.
\end{center}

\noindent Observe that, for each $k \in \{0,1,2\}$, the set $\{M^k_\ell[i]\ | \ \ell=0,1,2,3,4,5\}$ contains all  subdigraphs $F_j[i]$ of $D_i$  that contain the horizontal arc $(i_k, i_{k+1})$. It is straightforward to verify that $D_i$ contains no other elements. 

Given a type-I 2-factorization $(F_0, F_1, F_2, F_3)$, each set $\{F_0[i], F_1[i], F_2[i]\}$ corresponds to a unique triple of the form $\{M^{0}_{\ell_0}[i], M^{1}_{\ell_1}[i], M^{2}_{\ell_2}[i]\}$, where $\ell_0, \ell_1, \ell_2 \in \{0,1,2,3,4,5\}$. 

\begin{notation}
\label{not:Ki} \rm
By $\mathcal{K}_i=\{M^{0}_{\ell_0}[i], M^{1}_{\ell_1}[i], M^{2}_{\ell_2}[i]\}$, we denote the set of all  possible triples such that digraphs in each triple are pairwise arc-disjoint. 
\end{notation}

All elements of $\mathcal{K}_i$ are listed in the first column of Table \ref{fig:configurations}, given in Appendix \ref{appx}, and were obtain with the aid of a computer. 

For each $T \in \mathcal{K}_i$, the digraph  obtained by taking the union of all three digraphs in $T$, which we denote as $D$, contains 12 arcs. This means that there are exactly three arcs in $A(L_i)$ that are not contained in $A(D)$. It is routine to verify that these three arcs are non-horizontal and pairwise vertex-disjoint for each triple in $\mathcal{K}_i$, and can thus be described as a permutation in $S_3$. This set of three arcs corresponds to the digraph $F_3[i]$. In the third column of Table \ref{fig:configurations}, we give the permutation that describes $F_3[i]$.

Each type-I 2-factorization of $\vec{C}_n \wr \vec{C}_3$ can be described by a unique $n$-tuple $(t_0, t_1, \ldots, t_{n-1})$, where $t_i \in \mathcal{K}_i$.  Our objective is to show that no $n$-tuples  $(t_0, t_1, \ldots,$ $t_{n-1})$, where $t_i \in \mathcal{K}_i$, give rise to decomposition of $\vec{C}_n \wr \vec{C}_3$ into directed hamiltonian cycles when $n$ is even.

\begin{definition} \rm
Let $\mathcal{F}=(F_0, F_1, F_2, F_3)$ be a type-I 2-factorization of $\vec{C}_n \wr \vec{C}_3$. The $n$-tuple of elements $(t_0, t_1, \ldots, t_{n-1})$, where $t_i=\{F_0[i], F_1[i], F_2[i]\}$,  is the \textit{spine of $\mathcal{F}$}. 
\end{definition}

Our next objective is to describe each spine as an $n$-tuple of elements of a particular group constructed in Definition \ref{defn:group} below. To do so, we first introduce the following definition. 

\begin{definition} \rm
 \label{def:hor}
Let $\mathcal{F}=(F_0, F_1, F_2, F_3)$ be a type-I 2-factorization of $\vec{C}_n \wr \vec{C}_3$ with spine $(t_0, t_1, \ldots, t_{n-1})$.  In addition, for $i \in\mathds{Z}_n$ and $j \in \{0, 1,2\}$, let  $k_{(i,j)}$ be the element of $\mathds{Z}_3$ such that $F_j$ is the $(i, k_{(i,j)})$-rooted 2-factor. Observe that $\{k_{(i,0)}, k_{(i,1)}, k_{(i,2)}\}=\mathds{Z}_3$. For each $i \in\mathds{Z}_n$, we let 

\smallskip
{\centering
  $ \displaystyle
    \begin{aligned} 
&\sigma_{t_i}: \mathds{Z}_3 \mapsto \mathds{Z}_3, &(k_{(i,j)})^{\sigma_{t_i}}=k_{(i+1,j)} 
    \end{aligned}
    $
\par
\smallskip}

\noindent for all $j \in \{0,1,2\}$. We point out that $\{k_{(i+1,0)}, k_{(i+1,1)}, k_{(i+1,2)}\}=\mathds{Z}_3$ for all $i \in \mathds{Z}_n$ because $F_0, F_1$, and $F_2$ are arc-disjoint. Therefore, the function $\sigma_{t_i}$ is a permutation of $\mathds{Z}_3$ and is called the \textit{permutation of the horizontal arcs by $t_i$}.
\end{definition}

In the second column of Table \ref{fig:configurations}, we give the corresponding permutation of the horizontal arcs by $T_s$ for each $s \in \{1,2, \ldots, 24\}$. 

Using Definition \ref{def:hor}, we will now derive a necessary condition on $n$-tuples $(t_0, t_1, \ldots, t_{n-1})$, where $t_i \in \mathcal{K}_i$, that correspond to a 2-factorization of $\vec{C}_n \wr \vec{C}_3$. This necessary condition is given in Corollary \ref{cor:firststep}. First, we prove the following lemma. 

\begin{lemma}
\label{lem:phis}
Let  $\mathcal{F}=(F_0, F_1, F_2, F_3)$ be a type-I 2-factorization of $\vec{C}_n \wr \vec{C}_3$ with spine $(t_0, t_1, \ldots,$ $t_{n-1})$, $i \in\mathds{Z}_n$. For each $i \in\mathds{Z}_n$,  let  $\mu_i=\sigma_{t_0}\sigma_{t_1}\cdots\sigma_{t_{i}}$. Then, for $j \in \{0,1,2\}$, the type-1 2-factor $F_j$ is the $(i+1, j^{\mu_i})$-rooted 2-factor of $\mathcal{F}$. 
\end{lemma}

\begin{proof}
We prove the statement for $F_0$. Analogous reasoning applies to $F_1$ and $F_2$. For each $i\in \mathds{Z}_n$, $F_0$ is the $(i, k_{(i,0)})$-rooted 2-factor of $\mathcal{F}$. Note that $k_{(0,0)}=0$ since, by Assumption \ref{not:Fi},  $F_0$ is the $(0,0)$-rooted 2-factor of $\mathcal{F}$.  We prove the statement by way of induction on $i$. 

First, we consider the base case $j=1$. Then $\mu_0=\sigma_{t_0}$. By Definition \ref{def:hor}, we have that $0^{\sigma_{t_0}}=k_{(1,0)}$. The statement follows because $F_0$ is the $(1, k_{(1,0)})$-rooted 2-factor. Now, we assume that $F_0$ is the $(i, 0^{\mu_{i-1}})$-rooted 2-factor of $\mathcal{F}$.  By Definition \ref{def:hor}, this means that $0^{\mu_{{i-1}}}=k_{(i,0)}$. Therefore, we have $0^{\mu_i}=(0^{\mu_{i-1}})^{\sigma_{t_i}}=k_{(i,0)}^{\sigma_{t_i}}=k_{(i+1,0)}$.  The directed 2-factor $F_0$ is a $(i+1, 0^{\mu_i})$-rooted 2-factor, as claimed. \end{proof}

\begin{corollary}
\label{cor:firststep}
Let  $\mathcal{F}=(F_0, F_1, F_2, F_3)$ be a type-I 2-factorization of $\vec{C}_n \wr \vec{C}_3$ with spine  $(t_0, t_1, \ldots,$ $t_{n-1})$, and let $\mu=\sigma_{t_0}\sigma_{t_1}\cdots\sigma_{t_{n-1}}$. Then $\mu=id$. 
\end{corollary}

\begin{proof}
It follows from Lemma \ref{lem:phis} that, for $j \in \{0,1,2\}$,  $F_j$ is the $(0, j)$-rooted 2-factor and the $(0, j^{\mu})$-rooted 2-factor of $\mathcal{F}$. If $\mu\neq id$, then we have a contradiction. 
\end{proof}

Corollary \ref{cor:firststep} is the first step towards our main objective, which is to show that no $n$-tuples $(t_0, t_1, \ldots, t_{n-1})$ will give rise to a hamiltonian type-I 2-factorization of $\vec{C}_n \wr \vec{C}_3$ when $n$ is even. To that end, we now introduce another function which applies to any type-1 2-factor. 

\begin{definition} \rm
\label{defn:switch}
Let $\mathcal{F}=\{F_0, F_1, F_2, F_3\}$ be a directed 2-factorization of $\vec{C}_n \wr \vec{C}_3$. If $F_j\in \mathcal{F}$ is a type-1 2-factor and both dipaths of $F_j[i]$ are of length two, then $F_j$ is said to be \textit{switched at $i$}. We then define the following function:
$$
s_i[F_j]=
\begin{cases}
1,\ \textrm{if $F_j$ is switched at $i$};\\
0,\ \textrm{otherwise}.
\end{cases}
$$
\end{definition}

The type-1 2-factor $F_j$ is switched at $i$ precisely when the dipath of $F_j[i]$ that contains the horizontal arc of $C^i_3$ does not contain the horizontal arc of $C_3^{j+1}$. For a type-I 2-factorization $\mathcal{F}=(F_0, F_1, F_2, F_3)$ with $i\in \{0,1,2\}$, the type-1 2-factor $F_j$ is switched at $i$ if and only if 

\begin{center}
$F_j[i] \in \{ \rho^k(M^0_0[i]), \rho^k(M^0_2[i]), \rho^k(M^0_4[i])\ | \ k=0,1,2\}$. 
\end{center}

\noindent Refer to Figures \ref{fig:sw}, \ref{fig:sw1}, and \ref{fig:sw2} for an illustration of $M^0_0[i], M^0_2[i],$ and $M^0_4[i]$, respectively. 

We now use Definition \ref{defn:switch} to show that, if a type-I 2-factor of $\vec{C}_n\wr \vec{C}_3$ is a directed hamiltonian cycle, then we have the following necessary condition. 

\begin{lemma}
\label{lem:switch}
Let $\mathcal{F}=\{F_0, F_1, F_2, F_3\}$ be a directed 2-factorization of $\vec{C}_n \wr \vec{C}_3$. A type-1 2-factor $F_j\in \mathcal{F}$ is a directed hamiltonian cycle of $\vec{C}_n \wr \vec{C}_3$ only if $\sum_{j=0}^{n-1}s_i[F_j]$ is odd. 
\end{lemma}

\begin{proof}
Assume that $F_j$ is a type-1 2-factor that contains the arc $(0_i, 0_{i+1})$. If $F_j$ is a directed hamiltonian cycle, then it is not difficult to see that $F_j$ can be expressed as the concatenation of two dipaths $P_j$ and $Q_j$ such that
\begin{enumerate}[label=(C\arabic*)]
\item $F_j=P_jQ_j$;
\item $s(P_j)=0_j$, $t(P_j)=0_{j+2}$, and $P_j$ contains the horizontal arc $(0_j,  0_{j+1})$;
\item  $s(Q_j)=0_{j+2}$, $t(Q_j)=0_j$, and $0_{j+1} \not\in V(Q_j)$.
\end{enumerate}

\noindent Because $F_j$ is a type-1 2-factor, dipaths $P_j$ and $Q_j$ must jointly use exactly one horizontal arc from each $C^i_3$ for each $i \in\mathds{Z}_n$. By (C2), $P_j$ is the dipath that contains the horizontal arc of $C^0_3$. For $j \in \{0,1, \ldots, n-2\}$, if $F_j$ is switched at $i$ and $P_j$ contains a horizontal from $C^i_3$, then $Q_j$ must be the dipath that contains a horizontal arc from $C^{i+1}_3$ and vice versa.  

Suppose that $\sum_{j=0}^{n-1}s_i[F_j]$ is even. If $s_{n-1}[F_j]=1$, then $\sum_{j=0}^{n-2}s_i[F_j]$ is odd and $Q_j$ must be the dipath that contains the horizontal arc of $F_j$ from $C_3^{n-1}$. However, since $F_j$ is switched at $n-1$, and $F_j$ contains the arc $(0_i, 0_{i+1})$, then $t(P_j)=0_i$, a contradiction. Similarly, if $s_{n-1}[F_j]=0$, then $\sum_{j=0}^{n-2}s_i[F_j]$ is even and $P_j$ must be the dipath of $\mathcal{F}$ that contains the horizontal arc from $C_3^{n-1}$ and thus $t(P_j)=0_i$ since $F_j$ is not switched at $n-1$. In conclusion, it follows that $\sum_{j=0}^{n-1}s_i[F_j]$ is odd. \end{proof}

We will now describe each element of $\mathcal{K}_i$, defined in Notation \ref{not:Ki}, using an element of a group constructed in Definition \ref{defn:group} below. To construct this group, we will be using the semi-direct product of two particular groups. Therefore, we first define the semi-direct product of two groups in Definition \ref{defn:semi} below. 

\begin{definition}\rm 
\label{defn:semi} 
Let $\Gamma_1$ and $\Gamma_2$ be two groups and let $\phi: \Gamma_2 \mapsto \textrm{Aut}(\Gamma_1)$ where $\phi$ is a group homomorphism. We let $\phi(k)=\phi_k$ for each $k \in \Gamma_2$. The \textit{semi-direct product} of $\Gamma_1$ with $\Gamma_2$, denoted $\Gamma_1 \rtimes_{\phi} \Gamma_2$, is the group with elements $\{(h, k)\ | \ h\in \Gamma_1, k \in \Gamma_2\}$ and binary operation
\begin{center}
$(h, k)(h', k')=(h\phi_k(h'), kk')$. 
\end{center}
\end{definition}

\begin{definition}\rm
\label{defn:group}
Let $\Gamma_1=\mathds{Z}_2\oplus\mathds{Z}_2\oplus \mathds{Z}_2=\mathds{Z}_2^3$, where $\mathds{Z}_2$ is the additive group on two elements and $\oplus$ denotes the direct sum of two groups, and let $\Gamma_2=S_3 \oplus S_3$. Furthermore, for $(a_0, a_1, a_2) \in \mathds{Z}_2^3$ and $(\sigma, \gamma) \in S_3\oplus S_3$, we define $\phi: \Gamma_2 \mapsto \textrm{Aut}(\Gamma_1)$ by

\smallskip
{\centering
  $ \displaystyle
    \begin{aligned} 
&\phi_{(\sigma,~\gamma)}: \mathds{Z}_2^3\mapsto\mathds{Z}_2^3, \\
&\phi_{(\sigma,~\gamma)}(a_0, a_1, a_2)=(a_{0^{\sigma}}, a_{1^{\sigma}}, a_{2^{\sigma}}).
    \end{aligned}
    $
\par
\smallskip}

\noindent Then, we define $\Gamma_{\mathcal{F}} =\Gamma_1 \rtimes_{\phi} \Gamma_2$. 
\end{definition}

In Definition \ref{defn:group}, the permutation $\gamma$ acts trivially on $\Gamma_1$. It is then straightforward to verify that $\phi_{(\sigma, \gamma)} \in \textrm{Aut}( \Gamma_1)$ and that $\phi$ is a group homomorphism. 

\begin{remark} \rm If $g_a=((a_0, a_1, a_2), (\sigma_1, \gamma_1))$ and $g_b=((b_0, b_1, b_2), (\sigma_2, \gamma_2))$ such that $g_a, g_b \in \Gamma_{\mathcal{F}}$, then 
$g_ag_b= ((a_0+b_{0^{\sigma_1}}, a_1+b_{1^{\sigma_1}}, a_2+b_{2^{\sigma_1}}), (\sigma_1 \sigma_2, \gamma_1 \gamma_2)),$ with addition done in $\mathds{Z}_2$. 
\end{remark}

\begin{definition}\rm
\label{defn:elements}
Let $(F_0, F_1, F_2, F_3)$ be a type-I 2-factorization of $\vec{C}_n \wr \vec{C}_3$ with spine $(t_0, t_1, \ldots, t_{n-1})$. For each $i \in \mathds{Z}_m$, let $f_i: \mathcal{K}_i \mapsto \Gamma_{\mathcal{F}}$, such that 
$f_i(t_i)=((a_0, a_1, a_2), (\sigma_{t_i}, F_3[i]))$ where $F_3[i] \in S_3$ and $a_i=1$ if the $(j,i)$-rooted 2-factor is switched at $i$ or $a_i=0$ otherwise. 
\end{definition}

\begin{notation} \rm
\label{def:S} 
By $\mathcal{S}$, we denote the set of elements of $\Gamma_{\mathcal{F}}$ of the form $f_i(T)$ for some $T \in \mathcal{K}_i$. \end{notation}

In the fourth column of Table \ref{fig:configurations}, we give the element of $\mathcal{S}$ corresponding to each triple in $\mathcal{K}_i$. 

We now give necessary conditions for a type-I  2-factorization of $\vec{C}_n \wr \vec{C}_3$ with spine $(t_0, t_1, \ldots, t_{n-1})$, where $t_i \in \mathcal{K}_i$, to be a hamiltonian decomposition. 

\begin{lemma}
\label{prop:iff}
Let  $\mathcal{F}=(F_0, F_1, F_2, F_3)$ be a type-I 2-factorization of $\vec{C}_n \wr \vec{C}_3$ with spine  $(t_0, t_1, \ldots,$ $t_{n-1})$ such that
$(f_0(t_0), f_1(t_1), \ldots, f_{n-1}(t_{n-1}))=((h_0, k_0), (h_1, k_1), (h_2, k_2), \ldots, (h_{n-1}, k_{n-1}))$. Let $(g_1, g_2)=(h_0, k_0)(h_1, k_1)\cdots (h_{n-1}, k_{n-1})$. Then $\mathcal{F}$  is a hamiltonian decomposition of $\vec{C}_n \wr \vec{C}_3$ only if  $(g_1, g_2) \in \{((1,1,1), (id, (012))), ((1,1,1), (id, (021)))\}$.
\end{lemma}

\begin{proof}
We assume that $\mathcal{F}=(F_0, F_1, F_2, F_3)$ is a hamiltonian type-I 2-factorization of $\vec{C}_n \wr \vec{C}_3$. Let $F_3[i]=\gamma_i$, $k_i=(\sigma_{t_i}, \gamma_i)$, and let $g_2=(\mu, \gamma)\in S_3\oplus S_3$. We see that

\begin{center}
 $\mu=\sigma_{t_0}\sigma_{t_1}\cdots\sigma_{t_{n-1}}$, and $\gamma=\gamma_0\gamma_1 \cdots\gamma_{n-1}$. 
\end{center}

\noindent Corollary \ref{cor:firststep} implies that $\mathcal{F}$ is a 2-factorization of $\vec{C}_n \wr \vec{C}_3$ only if $\mu=id$. In addition, since $\gamma_i =F_3[i]$, it follows that $F_3$ is a directed hamiltonian cycle only if $\gamma$ is a permutation comprised of a single cycle. There are exactly two permutations in $S_3$ comprised of a single cycle. Therefore, $\gamma \in \{ (012), (021)\}$.  In conclusion, $g_2=(\mu, \gamma)\in \{(id, (012)), (id, (021))\}$. 

Next, we show that $g_1=(1,1,1)$. Let $h_i=(x^i_0, x^i_1, x^i_2)\in \mathds{Z}_2\oplus\mathds{Z}_2\oplus \mathds{Z}_2$, and let 
\begin{center}
$\mu_{i-1}=\sigma_{t_0}\sigma_{t_1}\cdots\sigma_{t_{i-1}}$ and $\omega_{i-1}=\gamma_0\gamma_1\dots \gamma_{i-1}$ 
\end{center}

\noindent where $i \in \{1,2, \ldots, n-1\}$.  We see that

\smallskip
{\centering
  $ \scriptstyle
    \begin{aligned} 
 g_1&=&&h_0+\sum_{j=1}^{n-1} \phi_{(\mu_{i-1}, \omega_{i-1})}(h_i)\\
 &=&&h_0+\phi_{(\mu_{0}, \omega_0)}(h_1)+\phi_{(\mu_{1}, \omega_1)}(h_2)+\phi_{(\mu_{2}, \omega_2)}(h_3)+\cdots+\phi_{(\mu_{n-2}, \omega_{n-2})}(h_{n-1}).
      \end{aligned}
  $ 
\par}
\smallskip

We claim that $\phi_{(\mu_{i-1}, \omega_{i-1})}(h_i)=(s_i[F_0], s_i[F_1], s_i[F_2])$. First, we point out that $\phi_{(\mu_{i-1}, \omega_{i-1})}(h_i)=(x^i_{0^{\mu_{i-1}}}, x^i_{1^{\mu_{i-1}}}, x^i_{2^{\mu_{i-1}}})$. Recall that, by Lemma \ref{lem:phis}, $F_j$ is the $(i, j^{\mu_{i-1}})$-rooted 2-factor of $\mathcal{D}$. If $F_j$ is switched at $i$, then $x^i_{j^{\mu_i}}=1$ by Definition \ref{defn:elements} and $s_i[F_j]=1$ by Definition \ref{defn:switch}. Consequently, $x^i_{j^{\mu_{i-1}}}=1=s_i[F_j]$. Otherwise, $x_{j^{\mu_{i-1}}}=0=s_i[F_j]$. It follows that $(x^i_{0^{\mu_{i-1}}}, x^i_{1^{\mu_{i-1}}}, x^i_{2^{\mu_{i-1}}})=(s_i[F_0], s_i[F_1], s_i[F_2])$, as desired. 

We then see that

\smallskip
{\centering
  $ \scriptstyle
    \begin{aligned} 
 g_1&=&&(x_0^0+\sum_{i=1}^{n-1} x_{0^{\mu_{i-1}}}^i, x_1^0+\sum_{i=1}^{n-1} x_{1^{\mu_{i-1}}}^i, x_2^0+\sum_{i=1}^{n-1} x_{2^{\mu_{i-1}}}^i)\\
	&=&&(x_0^0+\sum_{i=1}^{n-1} s_i[F_0], x_1^0+\sum_{i=1}^{n-1} s_i[F_1], x_2^0+\sum_{i=1}^{n-1} s_i[F_2])\\
	&=&&(\sum_{i=0}^{n-1} s_i[F_0], \sum_{i=0}^{n-1} s_i[F_1], \sum_{i=0}^{n-1} s_i[F_2]).\\
      \end{aligned}
  $ 
\par}
\smallskip

Since $F_0$, $F_1$, and $F_2$ are directed hamiltonian cycles, Lemma \ref{lem:switch} implies that $\sum_{i=0}^{n-1}s_i[F_j]$ is odd for each $j\in \{0, 1,2\}$. Because $g_1\in \mathds{Z}_2\oplus\mathds{Z}_2\oplus\mathds{Z}_2$, then $g_1=(1,1,1)$. In summary, $\mathcal{F}$ is a hamiltonian 2-factorization only if  $(g_1, g_2) \in \{ ( (1,1,1), (id, (012))), ((1,1,1), (id, (021)))\}$.\end{proof}
 
\begin{proposition}
\label{thm:typ1}
Let $n$ be an even integer. The digraph $\vec{C}_n \wr \vec{C}_3$ does not admit a type-I hamiltonian 2-factorization. 
\end{proposition}

\begin{proof}
Let $\mathcal{F}=(F_0, F_1, F_2, F_3)$ be a type-I 2-factorization of $\vec{C}_n \wr \vec{C}_3$ with spine  $(t_0, t_1, \ldots, t_{n-1})$ where $(f(t_0), f(t_1), \ldots, f(t_{n-1}))=((h_0, k_0), (h_1, k_1), (h_2, k_2), \ldots, (h_{n-1}, k_{n-1}))$. Note that $(h_i, k_i) \in \mathcal{S}$ for each $i \in\mathds{Z}_n$. Let $A=\{ ( (1,1,1), (id, (012))), ((1,1,1), (id, (021)))\}$.

First, we start with the case $n=2$. We compute the set of elements

\begin{center}
$U_2=\{ (h_0, k_0)(h_1, k_1)\ | \  (h_0, k_0), (h_1, k_1)\in \mathcal{S}\}$. 
\end{center}

Using GAP, we have evaluated all 576 products. The set $U_2$ contains 126 distinct elements and is disjoint from $A$. By Lemma \ref{prop:iff}, the digraph $\vec{C}_2\wr \vec{C}_3$ does not admit a type-I 2-factorization.   

Next, we consider the case $n=4$. It suffices to construct the set: 

\begin{center}
$U_4=\{ (h'_0, k'_0)(h'_1, k'_1)\ | \  (h'_0, k'_0), (h'_1, k'_1) \in U_2\}$. 
\end{center}

\noindent We see that $U_4$ is the set of the products of all pairs of elements of $U_2$. Again, using GAP, we have computed all 15876 products and have verified that $U_4$ and $A$ are disjoint. By Lemma \ref{prop:iff}, the digraph $\vec{C}_4\wr \vec{C}_3$ does not admit  a type-I 2-factorization.  

For $n\geqslant 6$ and $n$ even, we let

\begin{center}
$U_n=\{  (h'_0, k'_0)(h'_1, k'_1)\  | \   (h'_0, k'_0)\in U_2 \ \textrm{and} \ (h'_1, k'_1)\in U_{n-2}\}$. 
\end{center}

\noindent Our last step is to show that $U_n=U_4$ for all even $n\geqslant 4$. The elements of $U_6$ can be found  by computing all 18144 possible products of the form  $(h'_0, k'_0)(h'_1, k'_1)$, where  $(h'_0, k'_0) \in U_2$ and $(h'_1, k'_1) \in U_4$, using GAP. We then found that $U_4=U_6$. It follows that $U_4=U_6=U_n$ for all $n \geqslant 8$ and $n$ even. Since $U_6$ does not contain any element of $A$, Lemma \ref{prop:iff} implies that, for all even $n\geqslant 6$, $\vec{C}_n \wr \vec{C}_3$ does not admit a type-I 2-factorization that is also a hamiltonian decomposition. \end{proof}

\subsection{Type-II 2-factorization of $\vec{C}_n \wr \vec{C}_3$}
\label{sec:type2}

Our objective is to demonstrate that no type-II 2-factorization of $\vec{C}_n \wr \vec{C}_3$ is a hamiltonian decomposition. Our approach is considerably simpler than that of Subsection \ref{sec:fix}. We will begin by introducing the following assumption which is analogous to Assumption \ref{not:Fi}.

\begin{assumption}\rm
\label{not:Fi}
Let $\mathcal{F}=\{F_0, F_1, F_2, F_3\}$ be a type-II 2-factorization of $\vec{C}_n \wr \vec{C}_3$ comprised of four $2$-factors. We shall assume that $F_0$ is the type-2 2-factor of $\mathcal{F}$, $F_1$ the type-1 2-factor, and $F_2$ and $F_3$  the type-0 2-factors of $\mathcal{F}$. We then write $(F_0, F_1, F_2, F_3)$. 
\end{assumption}

In Figure \ref{fig:allfj.1},  we list all six possible pairs $(F_0[i], F_1[i])$ such that $F_1[i]$ contains the arc $(i_2, i_0)$. These six pairs form the set  $\{M_\ell[i]\ | \ \ell=0, 1, \ldots, 5\}$. Then, for each $ \ell \in \{1,2, \ldots, 6\}$, we let $M_{\ell+6k}[i]=\rho^k(M_\ell)$, where $\rho$ is defined in Definition \ref{defn:rhorho} and $k \in \{1,2\}$. This gives rise to a set of 18 pairs; it is routine to verify that no other pair exists. 

\begin{figure} [htpb!]
\begin{center}
\begin{subfigure}[c]{0.32 \textwidth}
\begin{tikzpicture}[
  very thick,
  every node/.style={circle,draw=black,fill=black!90}, inner sep=2]
    every node/.style={circle,draw=black,fill=black!90}, inner sep=1.5, scale=0.75]
  \node (x0) at (0.0,3.0) [label=above:$0$] [label=left:$i$] {};
  \node (x1) at (1.0,3.0) [label=above:$1$]{};
  \node (x2) at (2.0,3.0)[label=above:$2$] {};
  \node (x3) at (3.0,3.0)[draw=gray, fill=gray, label=above:$0$] {};
  \node (y0) at (0.0,2.0) [label=left:$i+1$] {};
  \node (y1) at (1.0,2.0) {};
  \node (y2) at (2.0,2.0) {};
  \node (y3) at (3.0,2.0)[draw=gray, fill=gray] {};
  
  \path [very thick, draw=color2, postaction={very thick, on each segment={mid arrow}}]
	(x0) to (x1)
	(x1) to  (x2)
	(x2) to (y2)
	(y2) to (y3)
	(y0) to (y1);
  \path [very thick, draw=color8, postaction={very thick, on each segment={mid arrow}}]
  	(x2) to (x3)
  	(x0) to (y0)
	(x1) to  (y1)
	(y1) to  (y2);
	
\end{tikzpicture}
\caption{The pair $M_{1}[i]$.}
\end{subfigure}
\begin{subfigure}[c]{0.32 \textwidth}
\begin{tikzpicture}[
  very thick,
  every node/.style={circle,draw=black,fill=black!90}, inner sep=2]
    every node/.style={circle,draw=black,fill=black!90}, inner sep=1.5, scale=0.75]
  \node (x0) at (0.0,3.0) [label=above:$0$] [label=left:$i$] {};
  \node (x1) at (1.0,3.0) [label=above:$1$]{};
  \node (x2) at (2.0,3.0)[label=above:$2$] {};
  \node (x3) at (3.0,3.0)[draw=gray, fill=gray, label=above:$0$] {};
  \node (y0) at (0.0,2.0) [label=left:$i+1$] {};
  \node (y1) at (1.0,2.0) {};
  \node (y2) at (2.0,2.0) {};
  \node (y3) at (3.0,2.0)[draw=gray, fill=gray] {};
  
  \path [very thick, draw=color2, postaction={very thick, on each segment={mid arrow}}]
	(x0) to (x1)
	(x1) to  (x2)
	(x2) to  (y3)
	(y0) to  (y1)
	(y1) to (y2);
  \path [very thick, draw=color8, postaction={very thick, on each segment={mid arrow}}]
  	(x2) to (x3)
  	(x0) to (y1)
	(x1) to  (y2)
	(y2) to  (y3);
	
\end{tikzpicture}
\caption{The pair $M_{2}[i]$.}
\end{subfigure}
\begin{subfigure}[c]{0.32 \textwidth}
\begin{tikzpicture}[
  very thick,
  every node/.style={circle,draw=black,fill=black!90}, inner sep=2]
    every node/.style={circle,draw=black,fill=black!90}, inner sep=1.5, scale=0.75]
  \node (x0) at (0.0,3.0) [label=above:$0$] [label=left:$i$] {};
  \node (x1) at (1.0,3.0) [label=above:$1$]{};
  \node (x2) at (2.0,3.0)[label=above:$2$] {};
  \node (x3) at (3.0,3.0)[draw=gray, fill=gray, label=above:$0$] {};
  \node (y0) at (0.0,2.0) [label=left:$i+1$] {};
  \node (y1) at (1.0,2.0) {};
  \node (y2) at (2.0,2.0) {};
  \node (y3) at (3.0,2.0)[draw=gray, fill=gray] {};
  \path [very thick, draw=color2, postaction={very thick, on each segment={mid arrow}}]
	(x0) to (x1)
	(x1) to  (x2)
	(x2) to (y1)
	(y1) to (y2)
	(y2) to (y3);
  \path [very thick, draw=color8, postaction={very thick, on each segment={mid arrow}}]
  	(x2) to (x3)
  	(x3) to (y2)
	(x1) to  (y0)
	(y0) to  (y1);

\end{tikzpicture}
\caption{The pair $M_{3}[i]$.}
\end{subfigure}

\vspace{0.5cm}

\begin{subfigure}[c]{0.32 \textwidth}
\begin{tikzpicture}[
  very thick,
  every node/.style={circle,draw=black,fill=black!90}, inner sep=2]
    every node/.style={circle,draw=black,fill=black!90}, inner sep=1.5, scale=0.75]
  \node (x0) at (0.0,3.0) [label=above:$0$] [label=left:$i$] {};
  \node (x1) at (1.0,3.0) [label=above:$1$]{};
  \node (x2) at (2.0,3.0)[label=above:$2$] {};
  \node (x3) at (3.0,3.0)[draw=gray, fill=gray, label=above:$0$] {};
  \node (y0) at (0.0,2.0) [label=left:$i+1$] {};
  \node (y1) at (1.0,2.0) {};
  \node (y2) at (2.0,2.0) {};
  \node (y3) at (3.0,2.0)[draw=gray, fill=gray] {};
  
  \path [very thick, draw=color2, postaction={very thick, on each segment={mid arrow}}]
	(x0) to (x1)
	(x1) to  (x2)
	(x2) to (y2)
	(y2) to (y3)
	(y0) to (y1);
  \path [very thick, draw=color8, postaction={very thick, on each segment={mid arrow}}]
  	(x2) to (x3)
  	(x0) to (y1)
	(x1) to  (y0)
	(y1) to  (y2);
	
\end{tikzpicture}
\caption{The pair $M_{4}[i]$.}
\end{subfigure}
\begin{subfigure}[c]{0.32 \textwidth}
\begin{tikzpicture}[
  very thick,
  every node/.style={circle,draw=black,fill=black!90}, inner sep=2]
    every node/.style={circle,draw=black,fill=black!90}, inner sep=1.5, scale=0.75]
  \node (x0) at (0.0,3.0) [label=above:$0$] [label=left:$i$] {};
  \node (x1) at (1.0,3.0) [label=above:$1$]{};
  \node (x2) at (2.0,3.0)[label=above:$2$] {};
  \node (x3) at (3.0,3.0)[draw=gray, fill=gray, label=above:$0$] {};
  \node (y0) at (0.0,2.0) [label=left:$i+1$] {};
  \node (y1) at (1.0,2.0) {};
  \node (y2) at (2.0,2.0) {};
  \node (y3) at (3.0,2.0)[draw=gray, fill=gray] {};
  
  \path [very thick, draw=color2, postaction={very thick, on each segment={mid arrow}}]
	(x0) to (x1)
	(x1) to  (x2)
	(x2) to  (y3)
	(y0) to  (y1)
	(y1) to (y2);
  \path [very thick, draw=color8, postaction={very thick, on each segment={mid arrow}}]
  	(x2) to (x3)
  	(x3) to (y2)
	(y2) to  (y3)
	(x1) to  (y1);	

\end{tikzpicture}
\caption{The pair $M_{5}[i]$.}
\end{subfigure}
\begin{subfigure}[c]{0.32 \textwidth}
\begin{tikzpicture}[
  very thick,
  every node/.style={circle,draw=black,fill=black!90}, inner sep=2]
    every node/.style={circle,draw=black,fill=black!90}, inner sep=1.5, scale=0.75]
  \node (x0) at (0.0,3.0) [label=above:$0$] [label=left:$i$] {};
  \node (x1) at (1.0,3.0) [label=above:$1$]{};
  \node (x2) at (2.0,3.0)[label=above:$2$] {};
  \node (x3) at (3.0,3.0)[draw=gray, fill=gray, label=above:$0$] {};
  \node (y0) at (0.0,2.0) [label=left:$i+1$] {};
  \node (y1) at (1.0,2.0) {};
  \node (y2) at (2.0,2.0) {};
  \node (y3) at (3.0,2.0)[draw=gray, fill=gray] {};
  \path [very thick, draw=color2, postaction={very thick, on each segment={mid arrow}}]
	(x0) to (x1)
	(x1) to  (x2)
	(x2) to (y1)
	(y1) to (y2)
	(y2) to (y3);
  \path [very thick, draw=color8, postaction={very thick, on each segment={mid arrow}}]
  	(x2) to (x3)
  	(x0) to (y0)
	(y0) to  (y1)
	(x1) to  (y2);
	
\end{tikzpicture}
\caption{The pair $M_{6}[i]$.}
\end{subfigure}
\end{center}
\caption{All possible pairs of digraphs $(F_0[i], F_1[i])$, where $F_0[i]$ is drawn in grey and $F_1[i]$ in black, such that $F_1[i]$ contains the arc $(i_2, i_0)$.}
\label{fig:allfj.1}
\end{figure}
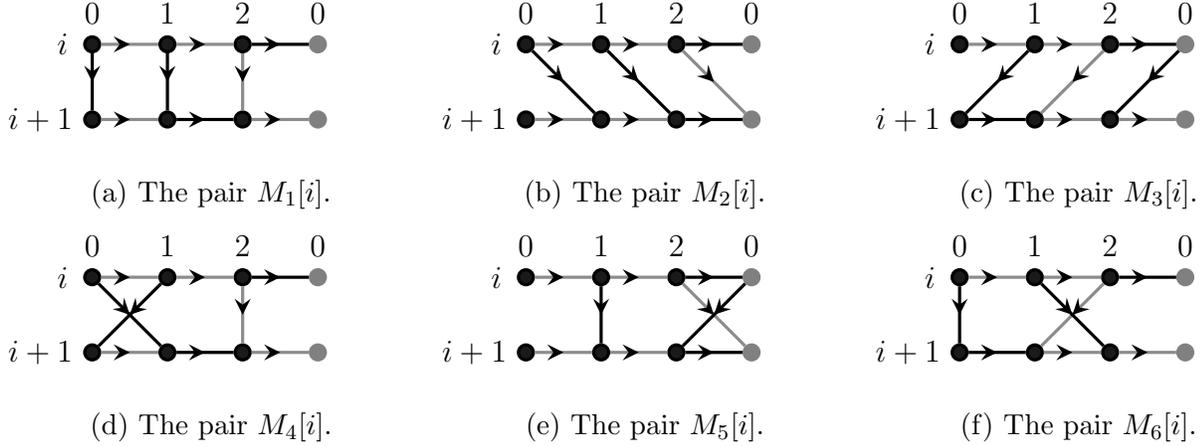

Observe that, if $(F_0[i], F_1[i])\in \{M_{\ell}[i]\ | \ \ell\equiv 1,2,3\ (\textrm{mod}\ 6)\}$, then the type-1 2-factor $F_1$ is switched at $i$, as defined in Definition \ref{defn:switch}. Otherwise, if $(F_0[i], F_1[i])\in \{M_{\ell}[i]\ | \ \ell\equiv 0,4,5\ (\textrm{mod}\ 6)\}$, then $F_1$ is not switched at $i$. Whether or not $F_1$ is switched as $i$ dictates the set of permutations that can be used to describe the corresponding type-0 2-factors $F_2$ and $F_3$, as stated in Lemma \ref{lem:sitch2} below. 

\begin{lemma}
\label{lem:sitch2}
Let $\mathcal{F}=(F_0, F_1, F_2, F_3)$ be a type-II directed 2-factorization of $\vec{C}_n\wr \vec{C}_3$. If $F_1$ is switched at $i$, then $F_2[i], F_3[i] \in \{(012), (021), id\}$. Otherwise, $F_2[i], F_3[i] \in \{(01), (02), (12)\}$. 
\end{lemma}

\begin{proof} For all possible pairs $F_0[i], F_1[i])$, there exist precisely six distinct sets of leftover arcs $L_t=A(L_i)-\left(A(F_0[i])\cup A(F_1[i]\right)$ where $t \in \{1,2, \ldots, 6\}$. This can be verified by using Figure \ref{fig:allfj.1}. 

Let $L_1$, $L_2$, and $L_3$ be the three sets of leftover arcs that correspond to a pair $(F_0[i], F_1[i])$ in which $F_1[i]$ is switched at $i$. Each of $L_1$, $L_2$, and $L_3$ admits exactly one partition into two sets of three vertex-disjoint arcs. Such a set of three disjoint arcs corresponds to a permutation in $\{(012), (021), id\}$. Consequently, $F_2[i], F_3[i] \in \{(012), (021), id\}$. 

Similarly, let $L_4$, $L_5$, and $L_6$ be the three sets of arcs that correspond to a pair $M_\ell=(F_0[i], F_1[i])$ in which $F_1[i]$ is not switched at $i$.  Each of $L_4$, $L_5$, and $L_6$ admits precisely one partition into two sets of three vertex-disjoint arcs. In this case, each such set of three arcs corresponds to a permutation in $\{(01), (02), (12)\}$. Hence, we see that $F_2[i], F_3[i] \in \{(01), (02), (12)\}$. \end{proof}

\begin{proposition}
\label{thm:typ2}
Let $n$ be an even integer. The digraph $\vec{C}_n \wr \vec{C}_3$ does not admit a type-II 2-factorization that is also a hamiltonian decomposition.   
\end{proposition}

\begin{proof} Let $\mathcal{F}=(F_0, F_1, F_2, F_3)$ be a type-II 2-factorization. In this proof, we will describe $F_2$ and $F_3$ as the following product of $n$ permutations of $S_3$ where $F_2[i]=\sigma_j$ and $F_3[i]=\mu_i$: $\sigma=\sigma_0\sigma_1\cdots \sigma_{n-1}$, and $\mu=\mu_0\mu_1\cdots \mu_{n-1}$. If $F_2$ and $F_3$ are directed hamiltonian cycles, then $\sigma, \mu \in \{(012), (021)\}$. Observe that $(012)$ and $(021)$ are even permutations. 

By Lemma \ref{lem:sitch2}, if $F_1$ is switched at $i$, then $\sigma_i, \mu_i \in \{(012), (021), id\}$. This means that $\sigma_i$ and $\mu_i$ are even permutations. Recall that, by Lemma \ref{lem:switch}, if $F_1$  is a type-1 directed hamiltonian cycle, then $\sum_{j=0}^{n-1}s_i(F_1)$ must be odd. This means that the set $\{\sigma_i \ | \ i \in \mathds{Z}_n\}$ contains an odd number of even permutations and thus, an odd number of odd permutations. Therefore, the permutation $\sigma$ is odd; likewise for $\mu$. This is a contradiction, meaning that $\mathcal{F}$ cannot be a hamiltonian decomposition. \end{proof}

\section{Acknowledgements}
The author would like to thank Mateja \v{S}ajna and Karen Meagher for their support and feedback on this paper which greatly improved its presentation.  The author was supported by the Natural Sciences and Engineering Research Council of Canada (NSERC) Post Graduate Scholarship program and the Pacific Institute for the Mathematical Sciences (PIMS).

\pagebreak

\appendix

\section{Supplemental material for Section \ref{sec:contra}}
\label{appx}
\begin{table} [ht!]
\begin{center}
\begin{tabular}{|p{0.65cm} |p{4.3cm}|p{1.4cm}|p{1.4cm}| p{5.30cm}|} 
 \hline
&Element of $\mathcal{K}_i$ & $\sigma_{T_s}$ & $F_3[i]$ & Element of $\Gamma_{\mathcal{F}}$ \\ \hline
$T_1$&$(M_0^0[i], M_5^1[i], M_4^2[i] )$ & $(0\,1)(2)$& $id$ &$((1, 0, 1),((0\,1)(2), id))$ \\ \hline
$T_2$&$(M_3^0[i], M_3^1[i], M_3^2[i])$ & $(0\,1\,2)$& $(0\,2\,1)$ &$((0, 0, 0),((0\,1\,2), (0\,2\,1)))$ \\ \hline
$T_3$&$(M_2^0[i], M_4^1[i], M_0^2[i])$ &$(0\,2)(1)$& $(1\,2)(0)$ &$((1, 1, 1),((0\,2)(1), (1\,2)(0)))$ \\ \hline
$T_4$&$(M_1^0[i], M_1^1[i], M_1^2[i])$ &$id$& $(0\,1\,2)$ &$((0, 0, 0),(id, (0\,1\,2)))$ \\ \hline
$T_5$&$(M_1^0[i], M_0^1[i], M_2^2[i] )$ &$(1\,2)(0)$& $(0\,1\,2)$ &$((0, 1, 1),((1\,2)(0), (0\,1\,2)))$ \\ \hline
$T_6$&$(M_3^0[i], M_3^1[i], M_0^2[i])$ &$(0\,1\,2)$& $(1\,2)(0)$ &$((0, 0, 1),((0\,1\,2), (1\,2)(0)))$ \\ \hline
$T_7$&$(M_1^0[i], M_1^1[i], M_4^2[i] )$ & $id$&$(0\,2)(1)$ &$((0, 0, 1),(id, (0\,2)(1)))$ \\ \hline
$T_8$&$(M_5^0[i], M_5^1[i], M_5^2[i])$ & $(0\,2\,1)$&$id$ &$((0, 0, 0),((0\,2\,1), id))$ \\ \hline
$T_9$&$(M_0^0[i], M_2^1[i], M_4^2[i] )$ &$(0\,1)(2)$& $(0\,2)(1)$ &$((1, 1, 1),((0\,1)(2), (0\,2)(1)))$ \\ \hline
$T_{10}$&$(M_0^0[i], M_3^1[i], M_3^2[i])$ &$(0\,1\,2)$& $(0\,2)(1)$ &$((1, 0, 0),((0\,1\,2), (0\,2)(1)))$ \\ \hline
$T_{11}$&$(M_5^0[i], M_2^1[i], M_5^2[i])$ &$(0\,2\,1)$& $(0\,2)(1)$ &$((0, 1, 0),((0\,2\,1), (0\,2)(1)))$ \\ \hline
$T_{12}$&$(M_3^0[i], M_2^1[i], M_4^2[i] )$ &$(0\,1)(2)$& $(0\,2\,1)$ &$((0, 1, 1),((0\,1)(2), (0\,2\,1)))$ \\ \hline
$T_{13}$&$(M_4^0[i], M_0^1[i], M_2^2[i] )$ &$(1\,2)(0)$& $(0\,1)(2)$ &$((1, 1, 1),((1\,2)(0), (0\,1)(2)))$ \\ \hline
$T_{14}$&$(M_0^0[i], M_2^1[i], M_1^2[i])$ &$(0\,1)(2)$& $(0\,1\,2)$ &$((1, 1, 0),((0\,1)(2), (0\,1\,2)))$ \\ \hline
$T_{15}$&$(M_4^0[i], M_3^1[i], M_2^2[i] )$ &$(1\,2)(0)$& $(0\,2\,1)$ &$((1, 0, 1),((1\,2)(0), (0\,2\,1)))$ \\ \hline
$T_{16}$&$(M_2^0[i], M_5^1[i], M_5^2[i])$ &$(0\,2\,1)$& $(1\,2)(0)$ &$((1, 0, 0),((0\,2\,1), (1\,2)(0)))$ \\ \hline
$T_{17}$&$(M_4^0[i], M_1^1[i], M_1^2[i])$ &$id$& $(0\,1)(2)$ &$((1, 0, 0),(id, (0\,1)(2)))$ \\ \hline
$T_{18}$&$(M_1^0[i], M_4^1[i], M_1^2[i])$ & $id$&$(1\,2)(0)$ &$((0, 1, 0),(id, (1\,2)(0)))$ \\ \hline
$T_{19}$&$(M_2^0[i], M_1^1[i], M_0^2[i])$ & $(0\,2)(1)$&$(0\,1\,2)$ &$((1, 0, 1),((0\,2)(1), (0\,1\,2)))$ \\ \hline
$T_{20}$&$(M_4^0[i], M_0^1[i], M_5^2[i])$ &$(1\,2)(0)$& $id$ &$((1, 1, 0),((1\,2)(0), id))$ \\ \hline
$T_{21}$&$(M_3^0[i], M_0^1[i], M_3^2[i])$ & $(0\,1\,2)$&$(0\,1)(2)$ &$((0, 1, 0),((0\,1\,2), (0\,1)(2)))$ \\ \hline
$T_{22}$&$(M_5^0[i], M_4^1[i], M_0^2[i])$ & $(0\,2)(1)$&$id$ &$((0, 1, 1),((0\,2)(1), id))$ \\ \hline
$T_{23}$&$(M_5^0[i], M_5^1[i], M_2^2[i] )$ &$(0\,2\,1)$& $(0\,1)(2)$ &$((0, 0, 1),((0\,2\,1), (0\,1)(2)))$ \\ \hline
$T_{24}$&$(M_2^0[i], M_4^1[i], M_3^2[i])$ &$(0\,2)(1)$& $(0\,2\,1)$ &$((1, 1, 0),((0\,2)(1), (0\,2\,1)))$ \\ \hline
\end{tabular}
\end{center}
\caption{All triples in the set $\mathcal{K}_j$.}
\label{fig:configurations}
\end{table}

\end{document}